\documentclass{amsart}

\usepackage{xspace,amsmath,mathrsfs,amscd}
\usepackage{amssymb}
\usepackage{hyperref}
\usepackage{color}
\usepackage{dirtytalk}

\newtheorem{theorem}{Theorem}[section]
\newtheorem{lemma}[theorem]{Lemma}

\theoremstyle{definition}

\newtheorem{example}[theorem]{Example}
\theoremstyle{remark}
\newtheorem{remark}[theorem]{Remark}

\numberwithin{equation}{section}

\let\bs\boldsymbol

\newcommand{\K}{\mathbb{K}}
\newcommand{\R}{\mathbb{R}}
\newcommand{\N}{\mathbb N}

\newcommand{\bD}{\mathbf D}
\newcommand{\eps}{\varepsilon}
\newcommand{\lraup}{\relbar\joinrel\rightharpoonup}
\newcommand{\llangle}{\langle\!\langle}
\newcommand{\rrangle}{\rangle\!\rangle}
\newcommand{\un}{\bs u_{1n}}
\newcommand{\unn}{\bs u_{2n}}

\begin{document}
	
	\title[]{\large Variational and Quasi-variational solutions 
		to thick flows}
	
	\author[]{Jos\'e Francisco Rodrigues}
	\address{CEMsUL -- Departamento de Matem\'atica, Faculdade de Ci\^encias, Universidade de Lisboa
		P-1749-016 Lisboa, Portugal}
	\email{jfrodrigues@ciencias.ulisboa.pt}
	
	\author[]{Lisa Santos} 
	\address{CMAT, Escola de Ci\^encias, Universidade do Minho, Campus de Gualtar, 4710-057 Braga, Portugal}
	\email{lisa@math.uminho.pt}
	
	\vspace{6mm}
	
	\begin{abstract}
		We formulate the flow of thick fluids as evolution variational and quasi-variational inequalities, with a variable threshold on the absolute value of the deformation rate tensor. In the variational case, we show the existence and uniqueness of strong and weak solutions in the viscous case and also the existence of strong and weak solutions in the inviscid case. These problems correspond to solve, respectively, the Navier-Stokes and the Euler equations with an additional generalised Lagrange multiplier associated with the threshold on the deformation rate tensor. Applying the continuous dependence of strong and weak solutions to the variational inequalities for the Navier-Stokes with constraints on the derivatives, and on their respective generalised Lagrange multipliers, we can solve the case of the variable threshold depending on the solution itself that correspond to quasi-variational problems.
		
		\vspace{2mm}
		
		$$ \text{Dedicated to Vsevolod Alekseevich Solonnikov, {\em in memoriam}}$$
	\end{abstract}
	
	\maketitle

	
	\section{Introduction}
	
	Several critical state physical problems lead to mathematical models, for instance in elastoplasticity, certain type-II superconductors, sandpiles growth, etc., with constraints of gradient type. Those models correspond to the class of unilateral problems in mathematical physics that can be formulated as variational inequalities (see, the recent survey \cite{RodriguesSantos2019} and its bibliography). In non-Newtonian flows, the limit case of shear thickening fluids, called thick fluids in \cite{Rodrigues2013}, where references can be found, the velocity vector field $\bs u=\bs u(x,t)$ is subjected to a constraint on the shear rate, i.e., its $\bD\bs u= \frac{1}{2}(\nabla \bs u+\nabla^T\bs u)$ satisfies the following bound 
	\begin{equation}
		\label{PSI}
		|\bD\bs u|\le\psi,
	\end{equation} 
	for a positive and bounded function  $\psi=\psi(x,t)$. The analysis and a numerical scheme for the stationary case of Stokes flow with $\psi=1$ was considered in \cite{DS2012}.  The more general case of a given function  $\psi(x,t)$ was considered for the classical Navier-Stokes equations in \cite{Rodrigues2013} and for a non-Newtonian variant in  \cite{MirandaRodrigues2014}. When the shear rate is strictly below the threshold $\psi$, the equations for the fluid flow are those classical equations. If the shear rate is attained, i.e., when $|\bD\bs u|=\psi$, the viscosity coefficient in general has a discontinuous change, corresponding to the existence of a Lagrange multiplier associated to the constraint \eqref{PSI} in the respective evolution variational inequality with a time dependent convex set of gradient type \cite{AzevedoRodriguesSantos2020}. In certain physical situations the threshold may implicitly depend on the solution itself, i.e., $\psi=\psi[\bs u]$ depends on $\bs u$ through some functional $\Psi$. Hence, the mathematical problem becomes an evolutionary quasi-variational inequality (QVI). The QVI case restricted to the Stokes flow of thick non-Newtonian fluids falls in the general framework of \cite{MirandaRodriguesSantos2020}, which Example 3.2.4 provides weak solutions by a Schauder fixed point argument. Also in the Stokes case, under certain restrictions, as in the framework of a contraction argument of Section 3.5 of \cite{RodriguesSantos2019}, it is possible to get unique strong quasi-variational solutions. In this work, we extend and develop those results to the general situation of convective flows of viscous thick fluids, as well as in the limit case of inviscid thick flows. We consider strong and weak solutions and we extend the existence of generalised Lagrange multipliers \cite{AzevedoRodriguesSantos2020}.
	
	In a bounded domain $\Omega$ of $\R^d$, $d\geq2$, with Lipschitz boundary $\partial\Omega$, the velocity field $\bs u=\bs u_\nu(x,t)$ of the incompressible thick fluid, with the viscosity $\nu>0$ or inviscid with $\nu=0$, divides, in general, the domain $Q_T=\Omega\times(0,T)$, for $T>0$, into two subregions  
	$$ \big\{(x,t): |\bD\bs u|<\psi[\bs u] \big\} \quad \text{ and } \quad \big\{(x,t): |\bD\bs u|=\psi[\bs u] \big\},$$
	satisfying formally the system
	\begin{align}
		\begin{cases}
			\label{IncompressibleDilatantFluidEq}
			&\partial _t\bs u -\nabla\cdot\big((\nu+\lambda_\nu[\bs u] )\bD\bs u-\bs u \otimes\bs u\big)+\nabla\pi  =\bs f,\\
			&\nabla\cdot\bs u  =0, \\
			&\lambda_\nu[\bs u]=0\ \text{ in }\big\{|\bD\bs u|<\psi[\bs u]\big\} \quad \text{ and } \quad\lambda_\nu[\bs u]\geq 0\text{ in }\big\{|\bD\bs u|=\psi[\bs u]\big\},
		\end{cases}
	\end{align}
	with initial condition $\bs u(x,0)=\bs u^0(x)$ in $\Omega$. 
	
	Here $\pi=\pi(x,t)$  denotes the pressure of the fluid and $\lambda_\nu[\bs u]=\lambda_\nu(x,t)$ is a generalised Lagrange multiplier. In the first subregion, were the shear rate $|\bD\bs u|$ is strictly below the threshold $\psi$, the velocity $\bs u$ solves the classical Navier-Stokes equation when $\nu>0$ and the Euler equation when $\nu=0$.
	
	Under general conditions, the solution of \eqref{IncompressibleDilatantFluidEq}
	is also a weak solution to the following  quasi-variational inequality, i.e. $\bs u=\bs u_\nu\in \mathscr V_2\cap C\big([0,T];\bs L^2(\Omega)\big)\cap \bs L^4(Q_T)$ is such that, for each $\nu\ge0$,
	\begin{equation} 
		\label{QIVf_i}
		\begin{cases}
			\bs u\in\ \mathscr K_{\psi[\bs u]},\\
			\displaystyle	\int_0^T\langle\partial_{t}\bs v,\bs v-\bs u\rangle
			+\nu\int_{Q_T}\bD\bs u:\bD(\bs v-\bs u)
			-\int_{Q_T}(\bs u\otimes\bs u):\nabla(\bs v-\bs u)\\
			\hfill{\displaystyle	\ge\int_{Q_T}\bs f\cdot(\bs v-\bs u)-\tfrac12\int_\Omega |\bs v(0)-\bs u^0|^2,\qquad \forall\,\bs v\in\mathscr K_{\psi[\bs u]}\cap\mathscr W_2 .}
		\end{cases}
	\end{equation}
	
	We denote vector functions and vector spaces of vector functions by bold symbols, in particular, for the usual Sobolev spaces $\bs W^{1,p}_0(\Omega)=W^{1,p}_0(\Omega)^d$, for $1\le p\le\infty$, and, whenever there exists no confusion, the norms in $L^p$ or $\bs L^p$ in $\Omega$ and in $Q_T=\Omega\times(0,T)$ will be simply denoted by $\|\,\cdot\,\|_{_p}$. We set, for $1\le p\le\infty,$
	\begin{equation*}\label{spaces}
		\mathbb H=\big\{\bs v\in\bs L^2(\Omega): \nabla\cdot \bs v=0\big\}, \qquad \mathbb V_p=\big\{\bs v\in\bs W^{1,p}_0(\Omega):\nabla\cdot \bs v=0\big\},
	\end{equation*}
	\begin{equation*}
		\mathscr V_p=\, L^p\big(0,T;\mathbb V_p\big), \qquad \mathscr V_\infty^\cap=L^\infty\big(0,T;\bigcap_{1<p<\infty}\mathbb V_p\big),
	\end{equation*}
	with $\langle\,\cdot\,,\,\cdot\,\rangle$ denoting the duality pairing between $\mathbb V_2'$ and  $\mathbb V_2$ and
	\begin{equation}\label{w2}
		\mathscr W_2=\big\{\bs v\in \mathscr V_2:\partial_t\bs v\in\mathscr V_2'\big\}\subset C\big([0,T];\bs L^2(\Omega)\big). 
	\end{equation}
	\begin{equation}\label{ConvexSetWeak}
		\mathscr K_\psi=\big\{\bs v\in \mathscr V_2:|\bD\bs v|\le\psi\text{ a.e. } (x,t)\in Q_T\big\}.
	\end{equation}
	
	When the strictly positive threshold $\psi$ has a bounded time derivative, in certain cases it is also possible to guarantee that the solution has a time derivative $\partial_{t}\bs u$ in $L^2(Q_T)$ and, therefore, it is also a strong solution to the quasi-variational inequality: to find 
	$\bs u\in \mathscr V_2\cap H^1\big(0,T;\mathbb H\big)\cap \bs L^4(Q_T)$, such that
	\begin{equation} \label{IVi}
		\begin{cases}
			\bs u =\bs u(t)\in\K_{\psi[\bs u](t)}\ \text{for}\ t\in(0,T),\ \bs u(0)=\bs u^0,\\
			\displaystyle\int_{\Omega}\partial_{t}\bs u\cdot(\bs w-\bs u)(t)
			+\nu\!\int_{\Omega}\bD\bs u:\bD(\bs w-\bs u)(t)
			-\displaystyle\int_{\Omega}(\bs u\otimes\bs u):\nabla(\bs w-\bs u)(t)\\
			\hfill{\displaystyle	\ge\int_{\Omega}\bs f\cdot(\bs w-\bs u)(t),\qquad\forall\,\bs w\in\K_{\psi[\bs u](t)},\ \text{a.e.}\ t\in(0,T)},
		\end{cases}
	\end{equation} 
	where, for a.e. $t\in(0,T)$, we define now the closed convex set for the threshold $\psi=\psi(x,t)$ by
	\begin{equation} \label{ConvexSet-t}
		\K_{\psi(t)}=\big\{\bs w\in\mathbb V_2:|\bD\bs w(x)|\leq\psi(x,t)\ \text{ a.e.} \ x\in \Omega\big\}.
	\end{equation}
	
	In this work, by solving these quasi-variational inequalities we obtain weak and strong solutions to thick flows. We start, in Section 2, by recalling and developing the results on strong  solutions to the variational inequality from \cite{Rodrigues2013}, i.e., when in \eqref{IVi} the smoother threshold is prescribed with a bounded time derivative and does not depend on the solution, in order to obtain explicit continuous dependence estimates for the case $\nu>0$. These estimates may be used to prove the existence of weak solutions and, with the fixed points arguments, to solve the quasi-variational inequalities, in particular, in defining explicit conditions on the data for the contraction case.
	
	The existence of weak solutions to the variational inequality with the additional convection term in both cases, $\nu>0$ and $\nu=0$, when the threshold $\psi$ is not differentiable in time, can be easily approximated by strong solutions with a suitable regularisation of the threshold. However, the proof of its uniqueness in the case $\nu>0$, which is also a new result and is presented in Section 3, requires an extra effort with respect to \cite{Kenmochi2013}, \cite{FukaoKenmochi2013} and \cite{MirandaRodriguesSantos2020}, on previous results on uniqueness of weak solutions to variational inequalities with time dependent convex sets. As a consequence, the explicit continuous dependence results with respect to the threshold are still valid for weak variational solutions to \eqref{QIVf_i} when $\Psi[\bs u]\equiv\psi$, with $\nu>0$, extending \cite{MirandaRodriguesSantos2020} to the case with the additional convection term when the threshold $\psi$ is only continuous in time and independent of the solution.
	
	In Section 4 we develop the results on multipliers for the evolution variational inequality \eqref{QIVf_i} with $\Psi[\bs u]\equiv\psi$, extending the results of \cite{AzevedoRodriguesSantos2020} to weak solutions with thresholds $\psi$ only continuous in time and with the additional convective term for variational solutions to thick flows in both situations, the viscous case $\nu>0$ and the inviscid case with $\nu=0$, which can be regarded as an asymptotic limit when $\nu\rightarrow0$. We observe that there is no singular perturbation in this asymptotic limit since the constraints imply the solutions are always in the space $\mathscr V_2=\, L^2\big(0,T;\mathbb V_2\big)$.
	
	Finally for the nondegenerate case $\nu>0$, in Section 5, we present some new results on the existence of weak solutions to the evolution quasi-variational inequality \eqref{QIVf_i} under a general assumption on the compactness of the functional relation $\Psi$ between the solution $\bs u$ and the threshold  $\psi[\bs u]$. We present two types of examples for $\Psi$, one with a general integral operators and another involving a linear parabolic equation with convection given by $\bs u$. Under a Lipchitz assumption on the functional $\Psi$, the explicit continuous dependence estimates on the variational solutions with respect to the variation of the threshold allow the application of a contraction argument to show the existence of unique weak and strong quasi-variational solutions, for small data or for small times with arbitrary data. Although these examples are borrowed from Section 3 of \cite{MirandaRodriguesSantos2020} and from paragraphs 3.5 and 3.6 of \cite{RodriguesSantos2019}, the results here are more general, because the additional nonlinear convective term $\bs u\otimes\bs u$.
	
	\section{Strong solutions to variational inequalities}
	
	We recall that, by Sobolev and Korn's inequalities, there exist $k_p$ and $K_p$ positive constants, depending only on $\Omega$ and $1<p<\infty$, such that
	\begin{equation}\label{PK}
		\|\bs v\|_{_q}\le k_p\|\nabla\bs v\|_{_p}\le K_p\|\bD\bs v\|_{_p} \qquad\forall\bs v\in\bs W^{1,p}_0(\Omega),
	\end{equation}
	where $q\in[1,p^*]$ being $p^*$ the critical Sobolev exponent, i.e.,  $p^*=dp/(d-p)$ if $p<d$, $p^*<\infty $ if $p=d$ and $p^*=\infty$ if $p>d$ (see, for instance, \cite{MNRR1996}).
	
	Let $\psi:Q_T\longrightarrow\R^+$ be a given function satisfying
	\begin{equation} \label{*}
		0<\psi_*\leq\psi(x,t)\leq \psi^* \quad\text{a.e. } (x,t)\in Q_T
	\end{equation}
	and
	\begin{equation} \label{psi}
		\psi\in W^{1,\infty}\big(0,T;L^\infty(\Omega)\big)
	\end{equation}
	or
	\begin{equation} \label{psiWeak}
		\psi\in C\big([0,T];L^\infty(\Omega)\big).
	\end{equation}
	In all this work we also assume that \begin{equation} \label{assumptions}
		\bs f\in\bs L^2(Q_{T})\quad\text{and}\quad \bs u^0\in\K_{\psi(0)},
	\end{equation}
	where the nonempty closed convex set $\K_{\psi(0)}$ is given by \eqref{ConvexSet-t} at $t=0$. We observe that, by \eqref{PK}, $\K_{\psi(t)}\subset\mathbb V_p$, for all $p<\infty$ and  for all $t\in[0,T]$, but not for $p=\infty$. Similarly, $\mathscr K_\psi\subset L^\infty\big(0,T;\mathbb V_p\big)$, for all $p<\infty$, i.e. $\mathscr K_\psi\subset \mathscr V_\infty^\cap $ but $\mathscr K_\psi\not\subset\mathscr V_\infty$.
	
	First we present a result on the existence of strong solution  for $\nu\ge0$.
	\begin{theorem}\label{1}
		Assuming \eqref{*}, \eqref{psi}, \eqref{assumptions} and for each $\nu\ge0$, there exists a strong solution $\bs u\in  H^1\big(0,T;\mathbb H\big)\cap \mathscr V_\infty^\cap$ of the evolution variational inequality
		\begin{equation} \label{IV}
			\begin{cases}
				\bs u =\bs u(t)\in\K_{\psi(t)}\ \text{for}\ t\in(0,T),\ \bs u(0)=\bs u^0,\\
				\displaystyle\int_{\Omega}\partial_{t}\bs u(t)\cdot\big(\bs w-\bs u(t)\big)
				+\nu\!\int_{\Omega}\bD\bs u:\bD\big(\bs w-\bs u(t)\big)-\displaystyle\int_{\Omega}\big(\bs u(t)\otimes\bs u(t)\big):\nabla\big(\bs w-\bs u(t)\big)\\
				\hspace{5cm}	\displaystyle	\ge\int_{\Omega}\bs f(t)\cdot\big(\bs w-\bs u(t)\big),\forall\,\bs w\in\K_{\psi(t)},\ \text{a.e.}\ t\in(0,T).
			\end{cases}
		\end{equation} 
		Besides, the solution is unique when $\nu>0$.
	\end{theorem}
	\begin{proof} We treat first the case $\nu>0$.
		
		We approximate the variational inequality using a regularization and penalization, letting  $0<\eps<1$. We start by fixing $q>\max\{d+1,4\}$ such that $\frac{q-1}2\in\N\setminus\{1\}$ and we define the scalar function $ k_\eps$  as follows: $ k_\eps=0$ when $s\le 0$, $ k_\eps=e^\frac{s}\eps-1$ when $0< s<\frac1\eps$ and $ k_\eps=e^\frac{1}{\eps^2}-1$ when $s\ge\frac1\eps$. We consider the approximating problem
		\begin{equation}\label{app}
			\begin{cases}
				\displaystyle\int_{\Omega}\partial_t\bs u_\eps(t)\cdot\bs\varphi+\int_\Omega  \Big(k_\eps\big(|\bD\bs u_\eps(t)|^2-\psi^2(t)\big)+\nu\Big)\bD\bs u_\eps(t):\bD\bs\varphi\\
				\hfill{\hspace{2cm}\displaystyle-\int_\Omega\big(\bs u_\eps(t)\otimes\bs u_\eps(t)\big):\nabla\bs\varphi+\eps\int_\Omega|\bD\bs u_\eps(t)|^{q-2}\bD\bs u_\eps(t):\bD\bs\varphi}\\
				\hfill{\displaystyle=\int_\Omega\bs f(t)\cdot\bs\varphi,\quad\forall\bs\varphi\in\mathbb V_2},\\
				\bs u_\eps(0)=\bs u^0\quad\text{ on }\Omega.
			\end{cases}
		\end{equation}
		The necessity of introducing the term $\eps|\bD\bs u_\eps(t)|^q\bD\bs u_\eps(t)$ in the definition of the approximating problem is related with the truncation of the real function $ k_\eps$, for $s\ge\frac1\eps$. This truncation implies that we lose the control of $|\bD\bs u_\eps|$ in the region $\{|\bD\bs u_\eps|^2>\psi^2+\frac1\eps\}$.
		
		Applying a known result (see \cite[Theorem 1.2, p. 162]{Lions1969}), problem \eqref{app} has a unique solution. We intend to obtain sufficient {\em a priori} estimates independent of $\eps$ that allow us, for each fixed $\nu>0$, to pass to the limit when $\eps\rightarrow0$. 
		
		Taking $\bs u_\eps(t)$ as test function in \eqref{app}, setting $\widehat k_\eps= k_\eps(|\bD\bs u_\eps(t)|^2-\psi^2(t))$ and $Q_t=\Omega\times(0,t)$, integrating in time  we have
		$$\tfrac12\int_\Omega|\bs u_\eps(t)|^2+\int_{Q_t}(\widehat k_\eps+\nu)|\bD\bs u_\eps|^2+\eps\int_{Q_t}|\bD\bs u_\eps|^q=\int_{Q_t}\bs f\cdot\bs u_\eps+\tfrac12\int_\Omega|\bs u^0|^2,$$
		from where we immediately obtain the estimates (being $C$ different positive constants independent of $\eps$ and $\nu$)
		\begin{equation}\label{est1}
			\|\bs u_\eps\|_{L^\infty(0,T;\bs L^2(\Omega))}\le C,\qquad \|\bD\bs u_\eps\|^q_{_q}\le\tfrac{C}{\eps}
		\end{equation}
		and 
		\begin{equation}\label{est2}
			\|\widehat k_\eps|\bD\bs u_\eps|^2\|_{_1}\le C,\qquad\|\widehat k_\eps\|_{_1}\le C,
		\end{equation}
		being the second estimate of \eqref{est2} a consequence of the first one, by noting that, as $\widehat k_\eps(\psi^2-|\bD\bs u_\eps|^2)\le0$, then
		\begin{align*}
			\int_{Q_T}\widehat k_\eps&\le\tfrac1{(\psi_*)^2}\int_{Q_T}\widehat k_\eps\psi^2=\tfrac1{(\psi_*)^2}\int_{Q_T}\widehat k_\eps(\psi^2-|\bD\bs u_\eps|^2)+\tfrac1{(\psi_*)^2}\int_{Q_T}\widehat k_\eps|\bD\bs u_\eps|^2\\
			&\le\tfrac1{(\psi_*)^2}\int_{Q_T}\widehat k_\eps|\bD\bs u_\eps|^2\le C.
		\end{align*}
		
		Setting $A_\eps=\{(x,t)\in Q_T:|\bD \bs u_\eps(x,t)|^2-\psi^2(x,t)\ge\frac1\eps\}$ and using \eqref{est2}, we get
		$$(e^\frac1{\eps^2}-1)|A_\eps|=\int_{A_\eps}\widehat k_\eps\le\int_{Q_T}\widehat k_\eps\le C,$$
		concluding that 
		\begin{equation}\label{maeps}
			|A_\eps|\le C/(e^\frac{1}{\eps^2}-1).
		\end{equation}
		But then, setting $r=\frac{q-1}2$,
		\begin{align*}
			\int_{Q_T}|\bD\bs u_\eps|^{q-1}&=\int_{Q_T\setminus A_\eps}|\bD\bs u_\eps|^{q-1}+\int_{A_\eps}|\bD\bs u_\eps|^{q-1}\\
			&\le 2^{r-1}\Big(\int_{Q_T\setminus A_\eps}\big(|\bD\bs u_\eps|^2-\psi^2\big)^{\frac{q-1}2}+\int_{Q_T\setminus A_\eps}\psi^{q-1}\Big)+\left(\int_{A_\eps}|\bD\bs u_\eps|^q\right)^{\frac{1}{q'}}|A_\eps|^\frac1q\\
			&\le 2^{r-1}\Big(r!\int_{Q_T}\widehat k_\eps+\|\psi\|^{q-1}_{_{q-1}}	\Big)+C\left(\tfrac{1}{\eps}\right)^\frac1{q'}/(e^{\frac{1}{\eps^{2}}}-1)^q,
		\end{align*}
		as, for $s\ge 0$, $e^s-1\ge \frac{s^r}{r!}$, and $C$ depends on the constants given in \eqref{est1} and \eqref{maeps} and then is independent of $\eps$ and $\nu$. So we have the additional estimate
		\begin{equation}\label{est3}
			\|\bD\bs u_\eps\|_{_{q-1}}\le C.
		\end{equation}

		Setting $m_\eps(s)=\displaystyle\int_0^sk_\eps(\tau)d\tau$, we can formally use $\partial_t\bs u_{\eps}$ as test function in \eqref{app}, integrating over $(0,t)$,
		\begin{align}\label{meps}
			\int_{Q_t}|\partial_t\bs u_\eps|^2+&\tfrac12\int_{Q_t}\partial_tm_{\eps}\big(|\bD\bs u_\eps(t)|^2-\psi^2(t)\big)+\tfrac\nu2\int_{Q_t}\partial_t|\bD\bs u_\eps(t)|^2\\
			\nonumber&    +\tfrac\eps{q}\int_{Q_t}\partial_t|\bD\bs u_\eps|^q
			=-\int_{Q_t}\widehat k_\eps\,\psi\,\partial_t\psi+\int_{Q_t}\big(\bs u_\eps\cdot\nabla\big)\bs u_\eps\cdot\partial_t\bs u_\eps+	\int_{Q_t}\bs f\cdot\partial_t\bs u_\eps\\
			\nonumber&\le \|\widehat k_\eps\|_{L^1(Q_t)}\|\psi\|^2_{W^{1,\infty}(0,T;L^\infty(\Omega))}+\Big(\|\bs u_\eps\|_{\frac{2(q-1)}{q-3}}\|\bD\bs u_\eps\|_{_{q-1}}+\|\bs f\|_{_2}\Big)\|\partial_t\bs u_\eps\|_{_2}
		\end{align}
		
		Recall that $m_\eps(s)=s$ if $s\le 0$ and $m_\eps(s)>0$ if $s>0$.  Therefore,
		\begin{align*}
			\int_{Q_t}\partial_tm_{\eps}\big(|\bD\bs u_\eps(t)|^2-\psi^2(t)\big)&=\int_\Omega m_{\eps}\big(|\bD\bs u_\eps(t)|^2-\psi^2(t)\big)-\int_\Omega m_{\eps}(|\bD\bs u^0|^2-\psi^2(0))\\
			&\ge\int_{\{|\bD\bs u_\eps(t)|^2-\psi^2(t)<0\} }\big(|\bD\bs u_\eps(t)|^2-\psi^2(t)\big)+\int_\Omega\big(\psi^2(0)-|\bD\bs u^0|^2\big)\\
			&\ge-\|\psi\|^2_{_\infty}-\|\bD\bs u^0\|^2_{_2}.
		\end{align*}
		Then, after applying Young's inequality to the last term of the inequality \eqref{meps}, we get
		\begin{multline}
			\label{dt}
			\int_{Q_T}|\partial_t\bs u_\eps|^2\le\tfrac{2\eps}{q}\|\bD\bs u^0\|^q_{_q}+ 2\|\widehat k_\eps\|_{_1}\|\psi\|^2_{W^{1,\infty}(0,T;L^\infty(\Omega))}\\
			+\Big(\|\bs u_\eps\|_{\frac{2(q-1)}{q-3}}\|\bD\bs u_\eps\|_{_{q-1}}+\|\bs f\|_{_2}\Big)^2+\big(\|\psi\|^2_{_\infty}+(1+\nu)\|\bD\bs u^0\|^2_{_2}\big)\le C,
		\end{multline}
		where $C$ is a constant independent of $\eps$ and $\nu<N$, using the estimates \eqref{est2} and \eqref{est3} and recalling that, as $q-1>d$, then $\bs u_\eps$ is bounded in $\bs L^\infty(Q_T)$ independently of $\eps$ and $\nu$.
		
		Observe that the set $\{\bs u_\eps:0<\eps<1\}$ is bounded in
		$$\mathcal W=\big\{\bs v\in L^{q-1}\big(0,T;\bs W^{1,q-1}(\Omega)\big): \partial_t\bs v\in \bs L^2(Q_T)\big\}$$
		and $\mathcal W$ is compactly included in $L^{q-1}\big(0,T;\bs C(\overline\Omega)\big)$. In fact, for $0<h<T$,
		\begin{align*}
			\|\bs u_\eps(t+h)-\bs u_\eps(t)\|^2_{\bs L^2(\Omega)}&=\int_\Omega\Big|\int_t^{t+h}\partial_t\bs u_\eps(\tau)d\tau\Big|^2 dx\le\int_\Omega h\int_t^{t+h}|\partial_t\bs u_\eps(\tau)|^2d\tau dx\\
			&\le h\|\partial_t\bs u_\eps\|_{_2}^2\le Ch.
		\end{align*}
		Then, setting $\tau_h\bs v=\bs v(t+h)$, we have
		$$\|\tau_h\bs u_\eps-\bs u_\eps\|_{L^{q-1}(0,T-h;\bs L^2(\Omega))}\le \sqrt{Ch}T^\frac1{q-1}\underset{h\rightarrow0^+}\longrightarrow0$$
		and, setting $\bs X=\bs W^{1,q-1}(\Omega)$, $\bs B=\bs C(\overline\Omega)$ and $\bs Y=\bs L^2(\Omega)$, as $\bs X$ is compactly included in $\bs B$ and $\bs B$ is continuously included in $\bs Y$,  then, by a result of Simon (see \cite[Theorem 5]{Simon1987}), $\{\bs u_\eps:0<\eps<1\}$ is compactly included in $L^{q-1}\big(0,T;\bs C(\overline\Omega))\big)$.
		
		It is now easy to follow the steps in \cite[Theorem 3.1]{Rodrigues2013} and obtain existence of solution to the variational inequality \eqref{IV}, when $\nu>0$. Essentially, by the estimates obtained above, we have the following convergences of a subsequence of $\{\bs u_\eps\}_\eps$, still denoted by $\{\bs u_\eps\}_\eps$:
		\begin{align}\label{convconv}
			&\partial_t \bs u_\eps\underset{\eps\rightarrow0}\lraup\partial_t\bs u\ \text{ in }\bs L^2(Q_T)\text{-weak},\quad\bD\bs u_\eps\underset{\eps\rightarrow0}\lraup\bD\bs u\ \text{ in }\bs L^{q-1}(Q_T)\text{-weak},\\
			\nonumber	&\bs u_\eps\underset{\eps\rightarrow0}\longrightarrow\bs u\ \text{ in }L^{q-1}\big(0,T;\bs C(\overline\Omega))\big).
		\end{align}
		
		Given $\bs w\in\K_{\psi(t)}$, and using $\bs w-\bs u_\eps$ as test function in the approximating problem \eqref{app}, the only difference with the reasoning done in \cite[Theorem 3.1]{Rodrigues2013} is what happens with the additional term $\eps|\bD \bs u_\eps|^{q-2}\bD\bs u_\eps$. But the estimate \eqref{est3} immediately  implies its strong convergence to zero in $\bs L^{q-1}(Q_T)$. 	We note that the last convergence in \eqref{convconv} is enough to pass to the limit when $\eps\rightarrow0$ in the term $\displaystyle\int_{Q_T}\bs u_\eps\otimes\bs u_\eps:\nabla(\bs w-\bs u_\eps)$, since $q-1\ge 4$.  
		The proof of existence of solution of the variational inequality \eqref{IV} is concluded observing that the proof, done in \cite{Rodrigues2013}, of $|\bD\bs u|\le\psi$ a.e. in $Q_T$ remains unchanged and therefore also $\bs u \in \mathscr V_\infty^\cap$. The proof of the uniqueness of the solution is also similar (see Theorem \ref{cd} below). 
		
		Now we treat the case of existence for $\nu=0$.
		
		Along the rest of the proof, $\bs u_\nu$ denotes the unique solution of \eqref{IV} with viscosity parameter $\nu>0$. Now, to avoid confusion, we denote $\bs u_\eps$ by $\bs u_{\nu\eps}$.
		We know from \eqref{dt} that $\|\partial_t\bs u_{\nu\eps}\|_{\bs L^2(Q_T)}\le C$, and so the same is true for $\|\partial_t\bs u_{\nu}\|_{\bs L^2(Q_T)}$. Besides, $|\bD\bs u_\nu|\le\psi^*$ and these estimates are sufficient to pass to the limit, for a subsequence of $\{\bs u_\nu\}_\nu$, when $\nu\rightarrow0$, obtaining a function $\bs u$ such that
		\begin{align}\label{limnu}
			&\partial_t\bs u_\nu\underset{\nu\rightarrow0}\lraup\partial_t\bs u \text{ in }\bs L^2(Q_T)\text{-weak},\qquad \bD\bs u_\nu\underset{\nu\rightarrow0}\lraup\bD\bs u \text{ in }\bs L^\infty(Q_T)\text{-weak*},\\
			\nonumber&\bs u_\nu\underset{\nu\rightarrow0}\longrightarrow\bs u \text{ in }\bs C(\overline Q_T).
		\end{align}
		The last convergence is also a consequence of  \cite[Theorem 5]{Simon1987}, now using the result with $p=\infty$.
		
		Note that, if $\omega$ is any measurable subset of $Q_T$, we have
		\begin{equation}\label{cl}
			\int_\omega|\bD\bs u|\le\varliminf_{\nu\rightarrow0}\int_\omega|\bD\bs u_\nu|\le\int_\omega\psi
		\end{equation}
		concluding that $|\bD\bs u|\le\psi$ a.e. in $Q_T$.
		
		It is now straightforward to see that, using the convergences in \eqref{limnu}, we can let $\nu\rightarrow0$ in the variational inequality \eqref{IV} integrated over $(s,t)$, with $0<s<t<T$, obtaining
		\begin{equation*}
			\int_s^{t}\int_{\Omega}\partial_{t}\bs u\cdot(\bs w-\bs u)
			-\int_s^{t}\int_{\Omega}(\bs u\otimes\bs u):\nabla(\bs w-\bs u)\ge\int_s^{t}\int_{\Omega}\bs f\cdot(\bs w-\bs u),
		\end{equation*}
		being $\bs w\in L^2\big(0,T;\mathbb V_2\big)$ such that $\bs w(\tau)\in\K_{\psi(\tau)}$, for a.e. $\tau\in(s,t)$. As $s$ and $t$ are arbitrary, we conclude  that $\bs u$ solves the variational inequality \eqref{IV} with $\nu=0$.
	\end{proof}
	
	In this work, when $\nu$ is fixed, we denote a solution $\bs u$ of the strong variational inequality \eqref{IV} with data $\bs f$, $\bs u^0$, $\psi$, by $S(\bs f,\bs u^0,\psi)$.
	
	We start with some {\em a priori} estimates.
	
	\begin{lemma}\label{Lapriori}
		Let $\nu\ge0$ and $\bs u=\bs u_\nu=S(\bs f,\bs u^0,\psi)$ be the strong solution of the variational inequality \eqref{IV}, under the assumptions \eqref{*}, \eqref{psi} and \eqref{assumptions}. Then
		\begin{equation}\label{apriori}
			\|\bs u\|_{_2}\le \sqrt T\|\bs u\|_{L^\infty(0,T;\bs L^2(\Omega))}\le  \sqrt Te^{\frac{T}2}\big(\|\bs f\|_{_2}^2+\|\bs u^0\|_{_2}^2\big)^\frac12.
		\end{equation}
		and
		\begin{equation}\label{infty}
			\|\bs u\|_{_\infty}\le K_p\psi^* |\Omega|^{1/p}= M,
		\end{equation}
		being $p>d$ and $K_p$ the constant in \eqref{PK}.
	\end{lemma}
	\begin{proof} 
		Using $\bf 0$ as test function in \eqref{IV}  and denoting $Q_t=\Omega\times(0,t)$, we obtain the estimate
		\begin{align}\label{apriori_extra}
			\tfrac12\int_\Omega|\bs u(t)|^2+\nu\int_{Q_t}|\bD\bs u|^2&\le\int_{Q_t}\bs f\cdot\bs u+\tfrac12\int_\Omega|\bs u^0|^2\\
			\nonumber	&\le \tfrac12\int_{Q_t}|\bs f|^2+\tfrac12\int_{Q_t}|\bs u|^2+\tfrac12\int_\Omega|\bs u^0|^2
		\end{align}
		and, applying Gr\"onwall's inequality ($y(t)\le y_0+\displaystyle\alpha \int_0^ty(s)ds$ with $\alpha=1$), we conclude \eqref{apriori}.
		
		As $\bs u\in L^\infty\big(0,T;\mathbb V_p(\Omega)\big)$ for all $1<p<\infty$, using the inequality \eqref{PK} with $p>d$, we have 
		\begin{equation}\label{u_linfty}
			\|\bs u(t)\|_{_\infty}\le K_p\|\bD\bs u(t)\|_{_p}\le K_p\|\psi(t)\|_{_p}\le K_p \psi^* |\Omega|^{1/p}
		\end{equation}
		for a.e. $t\in (0,T)$, by the assumption \eqref{*}, and the conclusion follows.
	\end{proof}
	
	\begin{remark}\label{lambda1} Taking into account the assumption \eqref{*}, the Poincaré inequality $\|\bs v\|_{_2}\le \sqrt{1/\kappa}\|\nabla\bs v\|_{_2}$, with $\kappa=\kappa(\Omega)>0$, and the equality $\|\nabla\bs v\|_{_2}=\sqrt{2} \|D\bs v\|_{_2}$, which are valid for all $\bs v\in \mathbb V_2(\Omega)$, we also have the following estimate for the solution $\bs u_\nu$ of \eqref{IV}, independently of $\bs f$ and $\bs u^0$,
		\begin{equation}\label{apriori2}
			\|\bs u\|_{_2}\le \sqrt{1/\kappa}\|\nabla\bs u\|_{_2}= \sqrt{2/\kappa}\|D\bs u\|_{_2}\le \psi^* \sqrt {2|\Omega|T/\kappa}.
		\end{equation}
	\end{remark}
	Strong continuous dependence results hold for the Navier-Stokes case $\nu>0$. For given data $(\bs f_i,\bs u^0_i,\psi_i)$, for $i=1,2$, both satisfying the assumptions \eqref{*} and \eqref{assumptions}, and using the estimates \eqref{apriori} and \eqref{infty} for the respective strong solutions, we  set the constants
	\begin{small}
		\begin{equation}\label{constants}
			\begin{cases}
				C_1=\tfrac1{\psi_*}\Big(2M^2|\Omega|+2\nu dT|\Omega|\big(\psi^*\big)^2 +M\sqrt {T|\Omega|} \big( \|\bs f_1\|_{_2} +\|\bs f_2\|_{_2}\big)\Big),\vspace{2mm}\\
				C_2=2dM\sqrt {2}|\Omega|T
				\tfrac{\psi^*}{\psi_*}\vspace{2mm}\quad C_3=dM^2/2,\\
				C_*=2(C_1+C_2),\quad C_\nu=1+2C_3/\nu.
			\end{cases}
		\end{equation}
	\end{small}
	
	\begin{theorem}\label{cd}
		For $i=1,2$ and $\nu>0$, let $\bs u_i=S(\bs f_i,\bs u^0_i,\psi_i)$ denote the unique solutions to the strong variational inequality \eqref{IV}. Under the assumptions \eqref{*}, \eqref{psi} and \eqref{assumptions}, the strong solutions $\bs u_i$ of \eqref{IV} satisfy, 
		\begin{equation} \label{Ff}
			\|\bs u_1-\bs u_2\|^2_{L^\infty(0,T;\bs L^2(\Omega))} \le e^{C_\nu T}\Big(\|\bs f_1-\bs f_2\|^2_{_2}+\|\bs u^0_1-\bs u^0_2\|^2_{_2}+C_*\|\psi_1-\psi_2\|_{_\infty}\Big)
		\end{equation}
		and,
		\begin{multline}\label{FF}
			\|\bs u_1-\bs u_2\|^2_{L^\infty(0,T;\bs L^2(\Omega))}+\nu	\|\bD(\bs u_1-\bs u_2)\|^2_{_2}\\
			\leq \big(1+C_\nu 
			T e^{C_\nu T}\big)\Big(\|\bs f_1-\bs f_2\|^2_{_2}
			+\|\bs u^0_1-\bs u^0_2\|^2_{_2}
			+C_*\|\psi_1-\psi_2\|_{_\infty}\Big).
		\end{multline}
	\end{theorem}
	\begin{proof}
		To prove \eqref{Ff} we will follow the calculations in \cite[Theorem 3.2]{Rodrigues2013}, detailing when necessary. 
		
		Setting, for $i,j=1,2$ and $j\neq i$, $\rho=\frac{\psi_*}{\psi_*+\beta}$, where $\psi_*$ is defined in \eqref{constants}, $\beta=\|\psi_1-\psi_2\|_{_\infty}$ and $\bs w_i(t)=\rho(t)\bs u_j(t)$, we observe that $\bs w_i(t)\in\K_{\psi_i(t)}$. Setting $\bs v=\bs u_1-\bs u_2$, we have
		\begin{equation}\label{vt}
			\tfrac12\tfrac{d\ }{dt}	\int_\Omega|\bs v(t)|^2+\nu\int_\Omega|\bD\bs v(t)|^2\\
			\le\tfrac12\int_\Omega|\bs f_1(t)-\bs f_2(t)|^2+\tfrac12\int_\Omega|\bs v(t)|^2+\Theta(t)+\Upsilon(t),
		\end{equation}
		where
		\begin{align*}
			\Theta(t)&=\int_\Omega\Big(\partial_t\bs u_1\cdot(\bs w_1-\bs u_2)+\nu\bD\bs u_1:\bD(\bs w_1-\bs u_2)
			+\bs f_1\cdot(\bs u_2-\bs w_1)+\partial_t\bs u_2\cdot(\bs w_2-\bs u_1)\\
			\nonumber&\quad+\nu\bD\bs u_2:\bD(\bs w_2-\bs u_1)+\bs f_2\cdot(\bs u_1-\bs w_2)\Big)(t)\\
			\nonumber&=(\rho-1)\int_\Omega\Big(\partial_t\big(\bs u_1\cdot\bs u_2\big)+\nu\bD\bs u_1:\bD\bs u_2+\nu\bD\bs u_2:\bD\bs u_1-\big(\bs f_1\cdot\bs u_2+\bs f_2\cdot\bs u_1\big)\Big)(t)
		\end{align*}
		and
		\begin{align*}
			\Upsilon(t)&=\int_\Omega\Big((\bs u_1\otimes\bs u_1):\nabla(\bs w_1-\bs u_1)+(\bs u_2\otimes\bs u_2):\nabla(\bs w_2-\bs u_2)\Big)\\
			&=\int_\Omega\big(\bs u_1\otimes \bs v+\bs v\otimes\bs u_2\big):\nabla\bs v
			+(\rho-1)\int_\Omega\Big((\bs u_1\otimes\bs u_1):\nabla\bs u_2+
			(\bs u_2\otimes\bs u_2):\nabla\bs u_1\Big)(t).
		\end{align*}
		
		Recalling that $|\rho-1|\le \frac{\beta}{\psi_*}=\frac{\|\psi_1-\psi_2\|_{_\infty}}{\psi_*}$ and using \eqref{infty}, we have the following estimate
		\begin{align*}
			\int_0^t\Theta(\tau) d\tau&\le \Big(2\|\bs u_1\|_{L^\infty(0,T;\bs L^2(\Omega))}\|\bs u_2\|_{L^\infty(0,T;\bs L^2(\Omega))}+2\nu d T|\Omega|	\|\bD\bs u_1\|_{_\infty}\|\bD\bs u_2\|_{_\infty}\\
			\nonumber&\quad+\|\bs f_1\|_{_2}\|\bs u_2\|_{_2}+\|\bs f_2\|_{_2}\|\bs u_1\|_{_2}\Big)\tfrac1{\psi_*}\|\psi_1-\psi_2\|_{_\infty}\\
			\nonumber	&\le\tfrac1{\psi_*}\Big(2M^2|\Omega|+2\nu dT|\Omega|\big(\psi^*\big)^2 +M\sqrt {T|\Omega|} \big( \|\bs f_1\|_{_2} +\|\bs f_2\|_{_2}\big)\Big)\|\psi_1-\psi_2\|_{_\infty}\\
			\nonumber	&=C_1\|\psi_1-\psi_2\|_{_\infty},
		\end{align*}
		with $C_1$ defined in \eqref{constants}.
		We also have
		\begin{align}\label{upsilon}
			\int_0^t\Upsilon(\tau)d\tau&\le\sqrt{d}M\Big(\int_{Q_t}|\bs v|^2\Big)^\frac12\Big(\int_{Q_t}|\bD\bs v|^2\Big)+2dM\sqrt {2}|\Omega|T
			\tfrac{\psi^*}{\psi_*}\|\psi_1-\psi_2\|_{_\infty}\\
			\nonumber&\le \tfrac\nu2\int_{Q_t}|\bD\bs v|^2+\tfrac{C_3}\nu\int_{Q_t}|\bs v|^2+C_2\|\psi_1-\psi_2\|_{_\infty},
		\end{align}
		where $C_2$ and $C_3$ are also defined in \eqref{constants}.
		
		Recalling \eqref{vt} and the estimates above, we obtain
		\begin{multline}\label{l2}
			\tfrac12	\int_\Omega|\bs v(t)|^2+\nu\int_{Q_t}|\bD\bs v|^2
			\le\tfrac12\int_{Q_t}|\bs f_1-\bs f_2|^2+\tfrac12\int_\Omega|\bs u^0_1-\bs u^0_2|^2\\
			+\big(\tfrac12+\tfrac{C_3}\nu\big)\int_{Q_t}|\bs v|^2+\tfrac\nu2\int_{Q_t}|\bD\bs v|^2+\big(C_1+C_2)\|\psi_1-\psi_2\|_{L^\infty(Q_t)}.
		\end{multline}
		Applying Gr\"onwall's inequality, we obtain for all $t\in(0,T)$,
		\begin{equation}\label{l3}
			\int_\Omega |\bs u_1-\bs u_2|^2(t) \le e^{C_\nu T}\Big(\|\bs f_1-\bs f_2\|^2_{_2}+\|\bs u^0_1-\bs u^0_2\|_{_2}^2\\
			+C_*\|\psi_1-\psi_2\|_{_\infty}\Big),
		\end{equation}
		obtaining \eqref{Ff}. Integrating \eqref{l3} in $(0,T)$ and using it in \eqref{l2} with the definitions \eqref{constants} we easily conclude \eqref{FF}.
	\end{proof}
	
	\section{Weak solutions to variational inequalities}\label{sec 3}
	
	Replacing the assumption \eqref{psi} by the weaker assumption \eqref{psiWeak} on the threshold $\psi$, we can easily extend the result on the existence of weak solution of \cite[Theorem 2.13]{MirandaRodriguesSantos2020}, with the additional convective term $\bs u\otimes\bs u:\nabla(\bs v-\bs u)$, which presents suplementary difficulties.
	\begin{theorem}\label{Ew}
		Under the assumptions \eqref{*}, \eqref{psiWeak}, \eqref{assumptions} and $\nu\geq 0$, there exists a weak solution $\bs u=\bs u_\nu\in C\big([0,T];\bs L^p(\Omega)\big)\cap\mathscr V_\infty^\cap$, for all $1<p<\infty,$ to the evolution variational inequality 
		\begin{equation} \label{IVf}
			\begin{cases}
				\bs u\in\ \mathscr K_{\psi},\\
				\displaystyle	\int_0^T\langle\partial_{t}\bs v,\bs v-\bs u\rangle
				+\nu\int_{Q_T}\bD\bs u:\bD(\bs v-\bs u)
				-\int_{Q_T}(\bs u\otimes\bs u):\nabla(\bs v-\bs u)\\
				\hfill{\displaystyle	\ge\int_{Q_T}\bs f\cdot(\bs v-\bs u)-\tfrac12\int_\Omega |\bs v(0)-\bs u^0|^2,\quad \forall\,\bs v\in\mathscr K_{\psi}\cap\mathscr W_2,}
			\end{cases}
		\end{equation} 
		where $\langle\,\cdot\,,\,\cdot\,\rangle$ is the duality pairing between $\mathbb V_2'$ and $\mathbb V_2$.
	\end{theorem}
	\begin{proof} Consider a sequence of constraints $\psi_n\in W^{1,\infty}\big(0,T;L^\infty(\Omega)\big)$ satisfying \eqref{*} and such that $\psi_n\underset{n}\rightarrow\psi$ in $C\big([0,T];L^\infty(\Omega)\big)$. 
		
		We treat first the case $\nu>0$.
		
		The corresponding sequence of unique solutions $\bs u_n=S(\nu,\bs f,\bs u^0,\psi_n)$ given by Theorem \ref{1}, corresponding to $\psi_n$ with the same $\bs f$ and $\bs u^0$, satisfy the estimate \eqref{FF} with $C_4=\sqrt{(1+C_\nu Te^{C_\nu T}))C_*}$
		\begin{equation}\label{FFn}
			\|\bs u_n-\bs u_m\|_{C([0,T];\bs L^2(\Omega))}+	\sqrt\nu\|\bD(\bs u_n-\bs u_m)\|_{_2}\leq \sqrt{2} C_4\|\psi_n-\psi_m\|^{1/2}_{_\infty}.
		\end{equation}
		
		Recalling \eqref{u_linfty}, we easily conclude that  $\bs u_n\in \mathscr K_{\psi_n} \cap C\big([0,T];\bs L^p(\Omega)\big)\cap \mathscr V_\infty^\cap$, for any $1<p<\infty$, and,
		consequently, there exists a $\bs u\in C\big([0,T];\bs L^p(\Omega)\big)\cap \mathscr V_\infty^\cap$ which is the limit of the Cauchy sequence $\bs u_n\underset{n}{\rightarrow}\bs u$ in $C\big([0,T];\bs L^p(\Omega)\big)\cap \mathscr V_p$, for any $1<p<\infty$. It is also clear that $\bs u\in \mathscr K_\psi$. 
		
		Define $\rho_n\underset{n}{\rightarrow}1$ by
		\begin{equation} \label{rho}
			\rho_n=\frac{\psi_*}{\psi_*+\eps_n},\ \text{ where } \eps_n=\|\psi_n-\psi\|_{_\infty}\underset{n}{\rightarrow} 0.
		\end{equation}
		
		For any $v\in\mathscr K_{\psi}\cap \mathscr W_2$, it is easy to see that $\bs v_n=\rho_n \bs v\in\mathscr K_{\psi_n}\cap \mathscr W_2$ and $\bs v_n\underset{n}{\rightarrow}\bs v$ in $C\big([0,T];\bs L^2(\Omega)\big)\cap \mathscr W_2$. As a strong solution is also a weak solution, $\bs u_n$ satisfies the evolution variational inequality 
		\begin{equation} \label{IVfn}
			\begin{cases}
				\bs u_n\in\ \mathscr K_{\psi_n},\\
				\displaystyle	\int_0^T\langle\partial_{t}\bs v_n,\bs v_n-\bs u_n\rangle
				+\nu\int_{Q_T}\bD\bs u_n:\bD(\bs v_n-\bs u_n)
				-\int_{Q_T}(\bs u_n\otimes\bs u_n):\nabla(\bs v_n-\bs u_n)\\
				\hfill{\displaystyle	\ge\int_{Q_T}\bs f\cdot(\bs v_n-\bs u_n)-\tfrac12\int_\Omega |\bs v_n(0)-\bs u^0|^2,\quad \forall\,\bs v_n\in\mathscr K_{\psi_n}\cap\mathscr W_2 .}
			\end{cases}
		\end{equation}
		Since it is clear that $\bs u_n\otimes\bs u_n\underset{n}{\rightarrow}\bs u\otimes\bs u$ in $\bs L^4(Q_T)$ and $\nabla v_n=\rho_n\nabla v\underset n\rightarrow\nabla v$ in $\bs L^2(Q_T)$, then we can pass to the limit in the inequality \eqref{IVfn}, concluding that $\bs u$ satisfies \eqref{IVf} and, therefore, is a weak solution to the evolution variational inequality. 
		
		Now we treat the case $\nu=0$.
		From the previous case, we know that for each $\nu>0$ we have a solution $\bs u_\nu$ of the variational inequality \eqref{IVf}.
		
		We select a strictly positive sequence $\{\nu_n\}_n$ that converges to zero and we consider the sequence of unique solutions $\bs u_n=S(\nu_n,\bs f,\bs u^0,\psi_n)$ of the strong variational inequality \eqref{IV}.
		
		Using an idea in the proof of Theorem 2.13 of \cite{MirandaRodriguesSantos2020}, we are going to show that the sequence $\{\bs u_n\}_n$ is relatively compact in $C\big([0,T];\bs L^p(\Omega)\big)$, $1< p<\infty$.
		
		Given $\eps>0$, there exists $\delta>0$ such that
		\begin{equation*}
			\eps_n=\sup_{|\tau-s|<\delta}\|\psi_n(\tau)-\psi_n(s)\|_{L^\infty(\Omega)}<\eps
		\end{equation*}
		for all $n$ large and all $\tau, s\in (0,T)$. Setting
		$$\zeta_n=\frac{\psi_*}{\psi_*+\eps_n},$$
		we have $\zeta_n\bs u_n\in\mathscr K_{\psi_n(\tau)}$, for all $\tau\in(s-\delta,s+\delta)$ and using $\zeta_n\bs u_n$ as a test function in \eqref{IV} for $t=\tau$, we get
		\begin{multline}\label{neta}
			\int_\Omega \partial_t\bs u_n(\tau)\cdot\big(\zeta_n\bs u_n(s)-\bs u_n(\tau)\big)+\nu_n\int_\Omega \bD\bs u_n(\tau):\bD\big(\zeta_n\bs u_n(s)-\bs u_n(\tau)\big)\\
			-\int_\Omega \big(\bs u_n(\tau)\otimes\bs u_n(\tau):\nabla\big(\zeta_n\bs u_n(s)-\bs u_n(\tau)\big)\ge\int_\Omega\bs f\cdot\big(\zeta_n\bs u_n(s)-\bs u_n(\tau)\big).
		\end{multline}
		
		Recall that $0<\zeta_n\le 1$ and, since $\bs u_n\in\mathscr K_{\psi_n}$, the sequence $\{\bs u_n\}_n$ is uniformly bounded in $L^\infty(Q_T)\cap L^\infty\big(0,T;\mathbb V_p\big),$ for any $p\in[1,\infty).$ Then, fixing $s$ and integrating in $\tau$ on $[s,t]$, we obtain
		\begin{align*}
			\tfrac12\int_\Omega|\bs u_n(t)-\bs u_n(s)|^2&=\tfrac12\int_s^t\tfrac{d\ }{d\tau}\int_\Omega|\bs u_n(\tau)-\bs u_n(s)|^2\\
			&=\int_s^t\int_\Omega\partial_\tau\bs u_n(\tau)\cdot\big((\bs u_n(\tau)-\zeta_n\bs u_n(s))+(\zeta_n-1)\bs u_n(s)\big)\\
			&\le(\zeta_n-1)\int_s^t\int_\Omega\partial_\tau\bs u_n(\tau)\cdot\bs u_n(s)+c_1|t-s|+c_2\int_s^t\|\bs f(\tau)\|_{\bs L^2(\Omega)}\\
			&\le (\zeta_n-1)\int_\Omega\big(\bs u_n(t)-\bs u_n(s)\big)\cdot\bs u_n(s)+c_1|t-s|+c_2|t-s|^\frac12\|\bs f\|_{\bs L^{2}(Q_T)}\\
			&\le c_3\eps+c_4|t-s|^\frac12,
		\end{align*}
		being the constants $c_3$ and $c_4$ consequently independent of $n$.
		
		Recalling \eqref{u_linfty}, given $p> 2$, we have
		$$\int_\Omega|\bs u_n(t)-\bs u_n(s)|^p=\int_\Omega|\bs u_n(t)-\bs u_n(s)|^{p-2}|\bs u_n(t)-\bs u_n(s)|^2\le (\psi^* \sqrt {2|\Omega|T/\kappa})^{p-2}\big(c_3\eps+c_4|t-s|^\frac12\big).$$
		Hence $\{\bs u_n\}_n$ is relatively compact in $C\big([0,T];\bs L^p(\Omega)\big)$, for any $p\in[1,\infty)$. So, fixing any $p>2$, there exists a subsequence, that we still denote by $\{\bs u_n\}_n$, that converges to a function $\bs u$ in $C\big([0,T];\bs L^p(\Omega)\big)$. Besides, as $\bs u_n\in\K_{\psi_n}$, then $|\bD\bs u_n|\le\psi^*$ and, also for a subsequence, we have $\{\bD\bs u_n\}_n$ converges to $\bD\bs u$ weakly in $\bs L^2(Q_T)$ and afterwards, using \eqref{PK}, we conclude that $\{\nabla\bs u_n\}_n$ also converges weakly in $\bs L^2(Q_T)$ to $\nabla\bs u$.
		
		Given $\bs v\in\mathscr K_\psi$ and $\rho_n$ defined in \eqref{rho},
		setting $\bs v_n=\rho_n\bs v$, as $\bs v\in\mathscr K_{\psi_n}\cap\mathscr V_2$ and  $\bs u_n$ also solves the weak variational inequality, then
		\begin{multline*}
			\int_0^T\langle\partial_{t}\bs v_n,\bs v_n-\bs u_n\rangle+\nu_n\int_{Q_T}\bD\bs u_n:\bD(\bs v_n-\bs u_n)\\
			-\int_{Q_T}(\bs u_n\otimes\bs u_n):\nabla(\bs v_n-\bs u_n)
			\ge\int_{Q_T}\bs f\cdot(\bs v_n-\bs u_n)-\tfrac12\int_\Omega |\bs v_n(0)-\bs u^0|^2.
		\end{multline*}
		By the observations above, we can pass to the limit when $n\rightarrow\infty$ in the above inequality, noting that the term $(\bs u_n\otimes\bs u_n):\nabla(\bs v_n-\bs u_n)$ converges to $(\bs u\otimes\bs u):\nabla(\bs v-\bs u)$ in $\bs L^1(Q_T)$, obtaining, as the term $\int_{Q_T}\bD\bs u_n:\bD(\bs v_n-\bs u_n)$ is bounded and $\nu_n\underset{n}{\rightarrow}0,$ that
		\begin{equation*}
			\int_0^T\langle\partial_{t}\bs v,\bs v-\bs u\rangle
			-\int_{Q_T}(\bs u\otimes\bs u):\nabla(\bs v-\bs u)
			\ge\int_{Q_T}\bs f\cdot(\bs v-\bs u)-\tfrac12\int_\Omega |\bs v(0)-\bs u^0|^2.
		\end{equation*}

		Recall that, arguing similarly as in \eqref{cl}, it is easy to verify that $\bs u\in\mathscr K_\psi$.
	\end{proof}  
	
	\begin{remark}
		Two weak solutions $\bs u_i=S(\nu,\bs f_i,\bs u^0_i,\psi_i)\in \mathscr K_{\psi_i}\cap C\big([0,T];\bs L^2(\Omega)\big)$, $i=1,2$, being Solutions Obtained as Limit of Approximations in $C\big([0,T];\bs L^2(\Omega)\big)\cap \mathscr V_2$, clearly also satisfy the continuous dependence estimates \eqref{Ff} and \eqref{FF}.
	\end{remark}
	
	We can prove the uniqueness and the continuous dependence estimates of weak solutions of the variational inequality \eqref{IVf} only in the nondegenerate case $\nu>0$, where we use the following result.
	
	\begin{lemma}\label{5.3}
		Let $\bs v\in C\big([0,T];\bs L^2(\Omega)\big)\cap \mathscr V_2\cap \mathscr K_\psi$, with $\psi \in C\big([0,T];L^\infty(\Omega)\big)$ and let  $\bs v(0)\in \mathbb K_{\psi(0)}$. Then there exist a regularizing sequence $\{\bs v_n\}_n$ and a sequence of scalar functions $\{\psi_n\}_n$ satisfying the following properties:
		\begin{enumerate}
			\item[i)] $\bs v_n\in C^1\big([0,T];\bs L^2(\Omega)\big)\cap H^1\big(0,T;\mathbb V_2\big)$;
			\item[ii)] $\bs v_n\underset{n}\rightarrow\bs v$ strongly in $C\big([0,T];\bs L^2(\Omega)\big)\cap \mathscr V_2$;
			\item[iii)] $\displaystyle\varlimsup_n\displaystyle\int_0^T\langle\partial_t\bs v_n,\bs v_n-\bs v\rangle\le0$;
			\item[iv)] $|\bD\bs v_n|\le\psi_n$, where $\psi_n\in C\big([0,T];L^\infty(\Omega)\big)$ and $\psi_n\underset{n}\rightarrow\psi$ in $C\big([0,T];L^\infty(\Omega)\big)$.
		\end{enumerate}
	\end{lemma}
	\begin{proof} The functions $\bs v_n$ are solutions of the problem
		\begin{equation}\label{regul}\begin{cases}
				\bs v_n+\tfrac1n\partial_t\bs v_n=\bs v,\\
				\bs v_n(0)=\bs v(0).
			\end{cases}
		\end{equation}
		The proof of i), ii), iii) is well known (see \cite[p. 274]{Lions1969} or \cite[ p. 206]{Roubicek2013}).
		
		Details of the proof of iv) can be found in \cite{MirandaRodriguesSantos2020}, were is also given
		\begin{equation}\label{psin}
			\psi_n(x,t)=e^{-nt}\int_0^t\psi(x,s)e^{ns}ds+e^{-nt}\psi(x,0).
		\end{equation}
	\end{proof}
	
	\begin{theorem}
		Assuming \eqref{*}, \eqref{psiWeak}, \eqref{assumptions} and $\nu>0$, the weak solution $\bs u\in \mathscr K_\psi\cap C\big([0,T];\bs L^2(\Omega)\big)$ to the evolution variational inequality \eqref{IVf} is unique. 
		
		Moreover, two weak solutions $\bs u_i$ of \eqref{IVf} corresponding to different data $\bs f_i, \bs u^i_0$ and $\psi_i, i=1,2,$ also satisfy the continuous dependence estimates \eqref{Ff} and \eqref{FF}.
	\end{theorem}
	\begin{proof} Suppose there exist two solutions, $\bs u_1,\bs u_2\in\mathscr K_\psi$, of the variational inequality\eqref{IVf}.
		
		Let $\{\bs u_{in}\}_n$, $i=1,2$, be the regularizing sequences of $\bs u_i$, both with initial condition $\bs u^0$ and observe that $\un$ and $\unn$ belong to $\mathscr K_{\psi_n}$, for $\psi_n$ defined in \eqref{psin}.
		
		Given $t\in(0,T)$ and $\theta\in(0,1)$, let $\delta\in(0,T-t)$ and set 
		$$\vartheta(s)=\theta\text{ if }s\in[0,t], \quad\vartheta(s)=\theta(1-\tfrac1\delta(s-t))\text{ if }s\in(t,t+\delta)\quad\text{ and}\quad\vartheta(s)=0\text{ if }s\in[t+\delta,T].$$
		
		Setting $\bs w=\bs u_1-\bs u_2$ and $\bs w_n=\un-\unn$, we observe that
		$$\bs w_n+\tfrac1n\partial_t\bs w_n=\bs w,\qquad\bs w_n(0)=0.$$
		
		Choosing $\psi_*$ a common positive lower bound of $\psi$ and $\psi_n$, we set
		\begin{equation*}
			\eps_n=\|\psi-\psi_n\|_{_\infty}\quad\text{and}\quad\rho_n=\frac{\psi_*}{\psi_*+\eps_n},
		\end{equation*}
		and we observe that $\rho_n\bs u_{in}\in\mathscr K_{\psi}$. Thus, we can easily verify that the functions
		\begin{equation*} 
			\bs v_{in}=\rho_n\big(\bs u_{in}+(-1)^i\vartheta\bs w_n\big)\in\mathscr K_\psi,\quad i=1,2.
		\end{equation*}
		So they can be used as test functions in the variational inequality \eqref{IVf}.
		
		Therefore, for $i=1,2$, we have
		\begin{multline*}
			\int_0^T\langle\partial_t \big(\rho_n(\bs u_{in}+(-1)^i\vartheta\bs w_n)\big),\bs u_i-\rho_n(\bs u_{in}+(-1)^i\vartheta\bs w_n)\rangle\\
			+\nu\int_{Q_T}\bD\bs u_i:\bD\big(\bs u_i-\rho_n(\bs u_{in}+(-1)^i\vartheta\bs w_n)\big)
			\le\int_{Q_T}(\bs u_i\otimes\bs u_i):\nabla\big(\bs u_i-\rho_n(\bs u_{in}+(-1)^i\vartheta\bs w_n)\big)\\
			+\int_{Q_T}\bs f\cdot\big(\bs u_i-\rho_n(\bs u_{in}+(-1)^i\vartheta\bs w_n)\big)+\tfrac12(1-\rho_n)^2\int_\Omega|\bs u^0|^2.
		\end{multline*}
		We sum both inequalities and write
		$$\bs u_i-\rho_n(\bs u_{in}+(-1)^i\vartheta\bs w_n)=(\bs u_i-\bs u_{in})+(1-\rho_n)\bs u_{in}+(-1)^{i+1}\vartheta\rho_n\bs w_n.$$
		
		\vspace{2mm}
		
		In order to estimate all the terms of the resulting inequality as $\bs u_{in}\underset{n}\rightarrow\bs u_i$, the terms in the left hand side from below and the other ones from above, we sum: 
		\begin{itemize}
			\item[--] all the parcels where appear the operator $\partial_t$, denoting this sum by  $A_n$;
			\item[--]  all the parcels where appear the operator $\bD$, denoting the sum by $B_n$;
			\item[--] all the parcels where appear the operator $\otimes$, denoting the sum by $C_n$;
			\item[--] all the parcels where appear the function $\bs f$, denoting  the sum by $D_n$;
		\end{itemize}
		so that
		\begin{equation}
			\label{abcd}
			A_n+B_n\le C_n+D_n+E_n,
		\end{equation}
		where
		\begin{equation}
			E_n=(1-\rho_n)^2\int_\Omega|\bs u^0|^2 \underset{n}\rightarrow 0.
		\end{equation}
		
		We start with the terms involving differentiation with respect to $t$.
		\begin{align*}
			A_n&=\rho_n\int_0^T\int_\Omega\tfrac1n\Big(|\partial_t\un|^2+|\partial_t\unn|^2-\tfrac\vartheta{n}|\partial_t\bs w_n^2|\Big)\\
			\nonumber	&\quad+\tfrac{\rho_n}2(1-\rho_n)\big(\|\un(T)\|_{_2}^2-\|\bs u^0\|_{_2}^2+\|\unn(T)\|_{_2}^2-\|\bs u^0\|_{_2}^2|\big)\\
			\nonumber	&\quad+\rho_n\int_0^T\int_\Omega-\vartheta'\big(\bs w_n\cdot\big(\bs w-\bs w_n\big)+(1-\rho_n)|\bs w_n|^2+2\vartheta\rho_n|\bs w_n|^2\big)\\
			\nonumber	&\quad+\rho_n\int_0^T\big(\rho_n\vartheta-(1-\rho_n)\vartheta-2\rho_n\vartheta^2\big)\langle\partial_t\bs w_n,\bs w_n\rangle,
		\end{align*}
		recalling that $\bs u_i-\bs u_{in}=\frac1n\partial_t\bs u_{in}$ and $\bs w-\bs w_n=\frac1n\partial_t\bs w_n$.
		
		Observe that
		$$\rho_n\int_0^T\int_\Omega\tfrac1n\Big(|\partial_t\un|^2+|\partial_t\unn|^2-\tfrac\vartheta{n}|\partial_t\bs w_n|^2\Big)\ge0$$
		as long as $0<\vartheta\le\frac12$. Moreover, as $\rho_n\underset{n}\rightarrow1$, we have
		$$\lim_n\Big(\tfrac{\rho_n}2(1-\rho_n)\big(\|\un(T)\|_{_2}^2-\|\bs u^0\|_{_2}^2+\|\unn(T)\|_{_2}^2-\|\bs u^0\|_{_2}^2\big)\Big)=0,$$
		$$\lim_n \rho_n\int_0^T\int_\Omega-\vartheta'\big(\bs w_n\cdot\big(\bs w-\bs w_n\big)+(1-\rho_n)|\bs w_n|^2=0\quad\text{and}\quad-2\rho_n^2\int_0^T\int_\Omega\vartheta'\vartheta|\bs w_n|^2\ge 0,$$
		since $\|\unn(T)\|_{_2}$, $\|\un(T)\|_{_2}$, $\|\bs w_n\|_{_2}$ are bounded independently of $n$ and $\bs w_n\underset{n}\rightarrow\bs w$ in $\bs L^2(Q_T)$. 
		
		Now we split the integral in $(0,T)$ with respect to $t$ of the remaining terms of $A_n$:
		
		(i) between $0$ and $t$, recalling that $\vartheta(s)=\theta$ for $s\in[0,t]$ and $\bs w_n(0)=\bs 0$,
		
		\begin{multline*}
			\tfrac{\rho_n}2\int_0^t\big((2\rho_n-1)\vartheta-2\rho_n\vartheta^2\big)\int_\Omega\partial_t|\bs w_n|^2(t)=\tfrac{\rho_n}2\big((2\rho_n-1)\theta
			\\
			-2\rho_n\theta^2\big)\|\bs w_n(t)\|_{_2}^2\underset{n}\rightarrow\theta(\tfrac12-\theta)\|\bs u_1(t)-\bs u_2(t)\|^2_{_2};
		\end{multline*}
		
		(ii) between $t$ and $t+\delta$, as $\vartheta'=-\tfrac\theta\delta$, $\vartheta(t+\delta)=0$ and $\vartheta(t)=\theta$,
		
		\begin{align*}
			\tfrac{\rho_n}2\int_t^{t+\delta}&\big((2\rho_n-1)\vartheta-2\rho_n\vartheta^2\big)\int_\Omega\partial_t|\bs w_n|^2=-\tfrac12\big((2\rho_n-1)\theta-2\rho_n\theta^2\big)\|\bs w_n(t)\|_{_2}^2\\
			&+\tfrac{\rho_n}2\int_t^{t+\delta}\big((2\rho_n-1)\tfrac\theta\delta-4\rho_n\vartheta(s)\tfrac\theta\delta\big)\|\bs w_n(s)\|_{_2}^2\\
			&\underset{n}\rightarrow -\theta(\tfrac12-\theta)\|\bs u_1(t)-\bs u_2(t)\|^2_{_2}+\tfrac\theta\delta\int_t^{t+\delta}\big(\tfrac12-2\vartheta(s)\big)\|\bs u_1(s)-\bs u_2(s)\|^2_{_2}\,ds;
		\end{align*}

		(iii) between $t+\delta$ and $T$, since $\vartheta(s)=0$ for $s>t+\delta$,
		\begin{equation*}
			\tfrac12\int_{t+\delta}^T\big((2\rho_n-1)\vartheta-2\rho_n\vartheta^2\big)\int_\Omega\partial_t|\bs w_n|^2=0.
		\end{equation*} 
		
		Gathering all the information about $A_n$, we conclude that
		\begin{equation*}
			\varliminf_n A_n\ge \tfrac\theta\delta\int_t^{t+\delta}\big(\tfrac12-2\vartheta(s)\big)\|\bs u_1(s)-\bs u_2(s)\|^2_{_2}\,ds
		\end{equation*}
		and, letting $\delta\rightarrow0^+$ afterwards, for $0<\vartheta\le\frac12$, we get for all $t\in(0,T)$:
		\begin{equation}\label{an}
			\varliminf_n A_n\ge \theta\big(\tfrac12-2\theta\big)\|\bs u_1(t)-\bs u_2(t)\|^2_{_2}.
		\end{equation}
		
		Now, we estimate $B_n$ also from below,
		\begin{align*}
			B_n&=\nu	\int_{Q_T}\bD\bs u_1:\Big(\bD(\bs u_1-\bs u_{1n})+(1-\rho_n)\bD\bs u_{1n}+\vartheta\rho_n\bD\bs w_n\Big)\\
			&\quad+	\nu\int_{Q_T}\bD\bs u_2:\Big(\bD(\bs u_2-\bs u_{2n})+(1-\rho_n)\bD\bs u_{2n} -\vartheta\rho_n\bD\bs w_n\Big)\\
			&\ge -\nu\Big(\|\bD\un\|_{_2}\|\bD(\bs u_1-\bs u_{1n})\|_{_2}+\|\bD\bs u_2\|_{_2}\|\bD(\bs u_2-\bs u_{2n})\|_{_2}\Big)\\
			&\quad-2(1-\rho_n)\nu(\psi^*)^2 +\rho_n\nu\int_{Q_T}\vartheta\bD(\bs u_1-\bs u_2):\bD\bs w_n.
		\end{align*}
		As, for $i=1,2$, $\|\bD \bs u_{in}\|_{_2}$ are bounded independently of $n$, ${\|\bD(\bs u_i-\bs u_{in})\|_{_2}}\underset{n}\rightarrow0$ and $\vartheta(s)=\theta$ for $0<s<t$, $\vartheta(s)=0$ for $s>t+\delta$, we have
		\begin{equation}\label{bn}
			\lim_n B_n\ge \nu\int_0^{t+\delta}\vartheta(s)\|\bD(\bs u_1(s)-\bs u_2(s))\|^2_{_2}ds\ge\theta\nu\int_0^t\|\bD(\bs u_1(s)-\bs u_2(s))\|^2_{_2}ds.
		\end{equation}
		
		We estimate $C_n$ from above,
		\begin{align*}
			C_n&=	\int_{Q_T}(\bs u_1\otimes\bs u_1):\nabla\Big((\bs u_1-\bs u_{1n})+(1-\rho_n)\bs u_{1n}+\vartheta\rho_n\bs w_n)\Big)\\
			&\quad+\int_{Q_T}(\bs u_2\otimes\bs u_2):\nabla\Big((\bs u_2-\bs u_{2n})+(1-\rho_n)\bs u_{2n}-\vartheta\rho_n\bs w_n)\Big)\\
			&\le d^2 M^2\sqrt{2}\Big(\|\bD(\bs u_1-\bs u_{1n})\|_{_2}+\|\bD(\bs u_2-\bs u_{2n})\|_{_2}\Big)+2(1-\rho_n)d^2M^2 \psi^*\sqrt{2|Q_T|}\\
			&\quad + \int_{Q_T}\vartheta\rho_n\Big((\bs u_1\otimes\bs u_1)-(\bs u_2\otimes\bs u_2)\Big):\nabla\bs w_n
		\end{align*}
		obtaining, as $\|\bD(\bs u_i-\bs u_{in})\|_{_2}\underset{n}\rightarrow0$ for $i=1,2$, $\rho_n\underset{n}\rightarrow1$ and $\vartheta(s)=0$ for $s>t+\delta$,
		\begin{align*}
			\lim_n C_n&\le	\int_0^{t+\delta}\int_\Omega\vartheta\Big((\bs u_1\otimes\bs u_1)-(\bs u_2\otimes\bs u_2)\Big):\nabla(\bs u_1-\bs u_2)\\
			&\le \theta\int_0^t\int_\Omega\Big((\bs u_1\otimes\bs u_1)-(\bs u_2\otimes\bs u_2)\Big):\nabla(\bs u_1-\bs u_2)+4M^2\sqrt 2\psi^*\delta\\
			&\le2\sqrt{2}M\theta\|\bs u_1-\bs u_2\|_{\bs L^2(Q_t)}\|\bD(\bs u_1-\bs u_2)\|_{\bs L^2(Q_t)}+4M^2\sqrt 2\psi^*\delta\\
			&\le\tfrac{\theta\nu}{2}\|\bD(\bs u_1-\bs u_2)\|_{\bs L^2(Q_t)}^2+\tfrac{4M^2\theta}\nu\|\bs u_1-\bs u_2\|^2_{\bs L^2(Q_t)}+4M^2\sqrt 2\psi^*\delta.
		\end{align*}
		and letting $\delta\rightarrow0^+$,
		\begin{equation}\label{cn}
			\lim_n C_n\le \tfrac{\theta\nu}{2}\|\bD(\bs u_1-\bs u_2)\|_{\bs L^2(Q_t)}^2+\tfrac{4M^2\theta}\nu\|\bs u_1-\bs u_2\|^2_{\bs L^2(Q_t)}.
		\end{equation}
		
		Finally,
		\begin{equation}\label{dn}
			\lim_n D_n=\lim_n \int_{Q_T}\bs f\cdot\big((\bs u_1-\bs u_{1n})+(1-\rho_n)\bs u_{1n}+(\bs u_2-\bs u_{2n})+(1-\rho_n)\bs u_{2n}\big)=0.	
		\end{equation}
		
		Therefore, taking the limit in \eqref{abcd}, using \eqref{an}, \eqref{bn}, \eqref{cn} we obtain, for all $t\in(0,T)$,
		\begin{multline*}
			\theta	\big(\tfrac12-2\theta\big)\|\bs u_1(t)-\bs u_2(t)\|^2_{_2}+\theta\nu\|\bD(\bs u_1-\bs u_2)\|^2_{\bs L^2(Q_t)}\\
			\le  \tfrac{\theta\nu}2\|\bD(\bs u_1-\bs u_2)\|_{\bs L^2(Q_t)}^2+\tfrac{4M^2\theta}\nu\|\bs u_1-\bs u_2\|^2_{\bs L^2(Q_t)}.
		\end{multline*}
		
		Dividing by $\theta$ and letting $\theta\rightarrow0$, we obtain
		\begin{equation*}
			\|\bs u_1(t)-\bs u_2(t)\|^2_{_2}\le\tfrac{ 8M^2}\nu\|\bs u_1-\bs u_2\|^2_{\bs L^2(Q_t)}
		\end{equation*}
		and, setting $y(t)=\|\bs u_1(t)-\bs u_2(t)\|^2_{_2}$, by using the Gr\"onwall's inequality, we obtain
		$$y(t)\le y(0)e^{\frac{8M^2}\nu t}=0$$
		concluding $\bs u_1=\bs u_2$ a.e. and therefore the uniqueness of weak solutions.
		
		To prove the continuous dependence estimates for $\nu>0$ fixed, we approximate the $\psi_i$, in $C\big([0,T];L^\infty(\Omega)\big)$, by sequences of functions $\psi_i^n\in W^{1,\infty}\big(0,T;L^\infty(\Omega)\big)$ as in Theorem \ref{Ew} and we set $\bs u_i^n=S(\bs f_i,\bs u^0_i,\psi_i^n)$, the unique solutions of the strong variational inequality \eqref{IV} with the respective data, as we have done in Lemma \ref{Lapriori}. Then, by Theorem \ref{cd}, we obtain
		\begin{equation}\label{ene}
			\|\bs u_1^n-\bs u_2^n\|_{L^\infty\big(0,T;\bs L^2(\Omega)\big)}\\
			\le  e^{C_\nu T}\big(\|\bs u^0_1-\bs u^0_2\|^2_{_2}+\|\bs f_1-\bs f_2\|^2_{_2}+C_*\|\psi_1^n-\psi_2^n\|_{_\infty}\big),
		\end{equation}
		were  $C_*$ and $C_\nu$ are defined in \eqref{constants} and are independent of $n$.
		
		Each sequence $\{\bs u_i^n\}$ is a Cauchy sequence in $C\big([0,T];\bs L^2(\Omega)\big)$ converging to $\bs u_i=S(\bs f_i,\bs u_i^0,\psi_i)$, the unique solution of the respective weak variational inequality \eqref{IVf} for $i=1,2$. Letting $n\rightarrow\infty$ in \eqref{ene}, we conclude the continuous dependence estimate \eqref{Ff} and, in a similar way, \eqref{FF} also holds for weak solutions.
	\end{proof}
	
	\section{Existence of generalised Lagrange multipliers}
	
	The strong Lagrange multiplier problem, associated to the strong solutions of the variational inequality for $\nu\ge0$ consists of finding  $(\lambda,\bs u)=(\lambda_\nu,\bs u_\nu)\in L^\infty(Q_T)'\times\big(\mathscr V_2\cap H^1\big(0,T;\bs L^2(\Omega)\big)\big)$ such that
	\begin{equation}\label{mlforte}
		\begin{cases}
			\displaystyle\int_{Q_T}\partial_t\bs u\cdot\bs v+\!\bs\llangle\lambda\bD\bs u,\bD\bs v\bs\rrangle+\nu\int_{Q_T}\bD\bs u:\bD\bs v\\
			\hfill{\displaystyle-\int_{Q_T}(\bs u\otimes\bs u):\nabla\bs v=\int_{Q_T}\bs f\cdot\bs v,\quad\forall\bs v\in \mathscr V_\infty}\\
			\bs u(0)=\bs u^0,\\
			|\bD\bs u|\le\psi\text{ in }Q_T,\quad\lambda\ge0\  \text{ and }\ \lambda\big(|\bD\bs u|-\psi\big)=0\  \text{ in }L^\infty(Q_T)'.
		\end{cases}
	\end{equation}
	Here $\bs\llangle\,\cdot\,,\,\cdot\,\bs\rrangle$ denotes the duality pairing between the Banach spaces $\bs L^\infty(Q_T)'$ and $\bs L^\infty(Q_T)$ and, given $\mu\in L^\infty(Q_T)'$ and $\bs\Xi\in \bs L^\infty(Q_T)$, $\mu\bs\Xi\in \bs L^\infty(Q_T)'$ is defined as follows:
	$$\bs\llangle\mu\,\bs\Xi,\bs\Phi\bs\rrangle=\langle\mu,\bs\Xi\cdot\bs\Phi\bs\rangle,\qquad\forall\Phi\in \bs L^\infty(Q_T),$$
	where $\langle\,\cdot\,,\,\cdot\,\rangle$, in the right hand side, is the duality pairing between $L^\infty(Q_T)'$ and $L^\infty(Q_T)$.

	We also define the weak Lagrange multiplier problem, associated to the weak solution of the variational inequality:    to find $(\lambda,\bs u)=(\lambda_\nu,\bs u_\nu)\in L^\infty(Q_T)'\times  \mathscr V_2$ such that
	\begin{equation}\label{mlfraco}
		\begin{cases}
			\displaystyle-\int_{Q_T}\bs u\cdot\partial_t\bs v+\bs\llangle\lambda\bD\bs u,\bD\bs v\bs\rrangle+\nu\int_{Q_T}\bD\bs u:\bD\bs v
			-\int_{Q_T}(\bs u\otimes\bs u):\nabla\bs v\\
			\hfill{\displaystyle=\int_{Q_T}\bs f\cdot\bs v+\int_\Omega\bs u^0\cdot\bs v(0),\quad\forall\bs v\in \mathscr V_\infty\cap\mathscr W_2,\  \bs v(T)=\bs 0,}
			\\			|\bD\bs u|\le\psi\text{ in } Q_T,\quad\lambda\ge0\ \text{ and }\ \lambda\big(|\bD\bs u|-\psi\big)=0\ \text{ in }L^\infty(Q_T)',
		\end{cases}
	\end{equation}
	where $\mathscr W_2$ is defined in \eqref{w2}.
	
	\begin{theorem}
		\label{tmlforte}
		Assume \eqref{*}, \eqref{psi} and \eqref{assumptions}. Then, for $\nu\ge0$, problem \eqref{mlforte} has a solution $(\lambda,\bs u)=(\lambda_\nu,\bs u_\nu)\in L^\infty(Q_T)'\times\big(H^1\big(0,T;\bs L^2(\Omega)\big)\cap\mathscr V_\infty^\cap\big)$. In addition, $\bs u$ is a strong solution of the variational inequality \eqref{IV}, which is unique in the case $\nu>0$.
	\end{theorem}
	\begin{proof}
		We prove first the existence of solution of problem \eqref{mlforte} when $\nu>0$, following the proof of Theorem 2.1 of \cite{AzevedoRodriguesSantos2020}.
		
		We consider the family of approximating problems \eqref{app} and we recall that the solution $\bs u_\eps\in \mathscr V_2\cap H^1\big(0,T;\bs L^2(\Omega)\big)$ and satisfies the {\em a priori} estimates \eqref{est1}, \eqref{est2} and \eqref{dt}.
		
		From the estimates \eqref{est2}, arguing as in Lemma 4.1 of \cite{AzevedoRodriguesSantos2020}, we get
		\begin{equation*}
			\|\widehat k_\eps\|_{L^\infty(\Omega)'}\leq\|\widehat k_\eps\|_{L^1(Q_T)}\le C,\qquad	\|\widehat k_\eps\bD \bs u_\eps\|_{\bs L^\infty(Q_T)'}\leq\|\widehat k_\eps\bD \bs u_\eps\|_{\bs L^1(Q_T)}\le C.
		\end{equation*}
		So, we can take a generalized subsequence, still denoted by $\eps\to 0$, a function $\bs u\in\mathscr V_2\cap H^1\big(0,T;\bs L^2(\Omega)\big)$, and, by the Banach-Alaoglu-Bourbaki theorem, there exist $\lambda\in L^\infty(Q_T)'$ and $\bs\Xi\in\bs L^\infty(Q_T)'$ with the following convergences:
		\begin{equation*}
			\bs u_\eps\underset{\eps}\rightarrow\bs u \text{ in }C\big([0,T];\bs L^2(\Omega)\big), \quad
			\bD\bs u_\eps\underset{\eps}\lraup\bD\bs u\text{ in }\bs L^2(Q_T), \quad\partial_t\bs u_\eps\underset{\eps}\lraup\partial_t\bs u\text{ in }\bs L^2(Q_T),
		\end{equation*}
		\begin{equation*}
			\widehat k_\eps\underset\eps\lraup\lambda\text{ in }L^\infty(Q_T)'\text{-weak*}\quad\text{ and }\quad\widehat k_\eps\bD\bs u_\eps \lraup\bs\Xi\text{ in }\bs L^\infty(Q_T)'\text{-weak*}.
		\end{equation*}
		Hence we have $\bs u(0)=\bs u^0$, and the uniform $L^1$ estimate on $\widehat k_\eps$ and the uniform $\bs L^2$ estimate on $\bD\bs u_\eps$ allow us to conclude that $|\bD\bs u|\le\psi,\text{ in } Q_T$, exactly as in Lemma 4.2 of \cite{AzevedoRodriguesSantos2020}. Therefore $\bs u\in L^\infty\big(0,T;\mathbb V_p\big)$, for all $1<p<\infty$.
		
		We can now let $\eps\rightarrow0$ in \eqref{app}, after integrating over $[0,T]$, obtaining, for any $\bs v\in \mathscr V_\infty$,
		\begin{equation*}
			\int_{Q_T}\partial_t\bs u\cdot\bs v+\bs\llangle\bs\Xi,\bD\bs v\bs\rrangle+\nu\int_{Q_T}\bD\bs u:\bD\bs v-\int_{Q_T}(\bs u\otimes\bs u):\nabla\bs v=\int_{Q_T}\bs f\cdot\bs v.
		\end{equation*}
		
		Exactly as in the proof of Theorem 2.1 of \cite{AzevedoRodriguesSantos2020} it follows 
		\begin{equation}
			\label{lambdaLambda}
			\bs\llangle\bs\Xi,\bD\bs u\bs\rrangle=\langle\lambda,|\bD\bs u|^2\rangle
		\end{equation}
		since the new term $\displaystyle\int_{Q_T}(\bs u\otimes\bs u:\nabla\bs u)=0$, which also implies the property
		\begin{equation*}
			\bs\llangle\bs\Xi,\bD\bs v\bs\rrangle=\bs\llangle\lambda\bD\bs u,\bD\bs v\bs\rrangle,\qquad\forall \bs v\in \mathscr V_\infty,
		\end{equation*}
		proving the variational equation of \eqref{mlforte} and the remaining equalities and inequalities in \eqref{mlforte} follow similarly.
		
		Now we consider the case $\nu=0$.
		
		Given $\nu>0$, we consider $(\lambda_\nu,\bs u_\nu)$ a solution of problem \eqref{mlforte} that is the limit of a generalized subsequence of $\widehat k_{\eps \nu}= k_\eps(|\bD\bs u_{\eps\nu}|^2-\psi^2)$, where $\bs u_{\eps\nu}$ is the solution of the approximating problem \eqref{app}. 
		
		We note first that
		$$\int_{Q_T}\partial_t\bs u_\nu\cdot\bs u_\nu+\langle\lambda_\nu,|\bD\bs u_\nu|^2\rangle+\nu\int_{Q_T}|\bD\bs u_\nu|^2=\int_{Q_T}\bs f\cdot\bs u_\nu+\int_\Omega|\bs u^0|^2,$$
		concluding that (see Remark \ref{lambda1}) 
		\begin{align*}
			\langle\lambda_\nu,|\bD\bs u_\nu|^2\rangle&\le\tfrac12\|\bs f\|^2_{_2}+\tfrac12\|\bs u_\nu\|^2_{_2}+\tfrac32\|\bs u^0\|_{_2}^2\\
			&\le \tfrac12\|\bs f\|^2_{_2}+(\psi^*)^2|\Omega|T/\lambda_1+\tfrac32\|\bs u^0\|^2_{_2}:=C_1.
		\end{align*}
		Following \cite{AzevedoRodriguesSantos2024},
		\begin{align*}
			\|\lambda_\nu\|_{L^\infty(Q_T)'}&=\sup_{\|\zeta\|_{L^\infty(Q_T)}\le 1}\big|\langle\lambda_\nu,\zeta\rangle\big|\le \sup_{\|\zeta\|_{L^\infty(Q_T)}\le 1}\langle\lambda_\nu,|\zeta|\rangle\le\langle\lambda_\nu,1\rangle\le\langle\lambda_\nu,\tfrac{\psi^2}{(\psi_*)^2}\rangle\\
			\nonumber&\le\tfrac1{(\psi_*)^2}\langle\lambda_\nu,|\bD\bs u_\nu|^2\rangle\le\tfrac{C_1}{(\psi_*)^2}.
		\end{align*}
		But then, setting $\bs\Xi_\nu=\lambda_\nu\bD\bs u_\nu$, we have the estimate
		\begin{equation*}
			\|\bs\Xi_\nu\|_{\bs L^\infty(Q_T)'}=\sup_{\|\bs\xi\|_{\bs L^\infty(Q_T)}\le 1}\big|\langle\lambda_\nu,\bD\bs u_\nu\cdot\bs\xi\rangle\big|\le\|\lambda_\nu\|_{L^\infty(Q_T)'}\|\bD\bs u_\nu\|_{\bs L^\infty(Q_T)}\le \tfrac{C_1}{(\psi_*)^2}\psi^*,
		\end{equation*}	
		as $|\bD\bs u_\nu|\le\psi^*$. By the estimate \eqref{dt}, we have
		$$\|\partial_t\bs u_\nu\|_{\bs L^2(Q_T)}\le\varliminf_{\eps\rightarrow0}\|\partial_t\bs u_{\eps,\nu}\|_{\bs L^2(Q_T)}\le C,$$
		where $C$ is a positive constant independent of $\eps$ and $\nu$.
		
		So, with the same arguments as in the proof of Theorem \ref{1}, we conclude that for a generalised subsequence
		\begin{align}\label{limiteLM}
			&	\partial_t\bs u_\nu\underset{\nu\rightarrow0}\lraup\partial_t\bs u\ \text{ in }\bs L^2(Q_T)\text{-weak},\quad\bD\bs u_\nu\underset{\nu\rightarrow0}\lraup\bD\bs u\ \text{ in }\bs L^\infty(Q_T)\text{-weak*},\\
			\nonumber&\bs u_\nu\underset{\nu\rightarrow0}\longrightarrow\bs u\ \text{ in }C(\overline Q_T),\\
			\nonumber&\lambda_\nu\underset{\nu\rightarrow0}\lraup\lambda\ \text{ in }L^\infty(Q_T)'\text{-weak*},\quad \bs\Xi_\nu\underset{\nu\rightarrow0}\lraup\bs\Xi\ \text{ in }\bs L^\infty(Q_T)'\text{-weak*},
		\end{align}
		being the two last convergences also a consequence of the Banach-Alaoglu-Bourbaki theorem. And $\bs u_\nu\underset{\nu\rightarrow0}\longrightarrow\bs u$  in $C(\overline Q_T)$ because this generalised sequence is bounded in $\{\bs v\in L^\infty\big(0,T;\bs W^{1,\infty}(\Omega)\big):\partial_t\bs v\in \bs L^2(Q_T)\}$ and this set is compactly included in $C(\overline Q_T)$ by \cite[Theorem 5]{Simon1987}.
		This is enough to pass to the limit, when $\nu\rightarrow0$ in \eqref{mlforte}, obtaining
		$$\int_{Q_T}\partial_t\bs u\cdot\bs v+\bs\llangle\bs\Xi,\bD\bs v\bs\rrangle-\int_{Q_T}(\bs u\otimes\bs u):\nabla\bs v
		=\int_{Q_T}\bs f\cdot\bs v,\qquad\forall\bs v\in \mathscr V_\infty.$$
		
		Observe that, as $\lambda_\nu\ge0$, we also have $\lambda\ge0$. Given $\omega$ any measurable subset of $Q_T$, let $\bs \xi\in\bs L^1(Q_T)$ be defined by $\bs\xi=\frac{\bD\bs u}{|\bD\bs u|}$ in $\omega\cap\{|\bD\bs u|\neq0\}$ and $\bs\xi=0$ otherwise. Then
		$$\int_\omega\psi\ge\int_\omega|\bD\bs u_\nu|\ge\int_\omega \bD\bs u_\nu\cdot\bs\xi\underset{\nu\rightarrow0}\longrightarrow\int_\omega \bD\bs u\cdot\xi=\int_\omega|\bD\bs u|$$
		and so $|\bD\bs u|\le\psi$ a.e. in $Q_T$.
		
		It remains to show that $\lambda\bD\bs u=\bs\Xi$ and that $\lambda(|\bD\bs u|-\psi)=0$.  But, recalling that $\langle\lambda_\nu,|\bD\bs u_\nu|^2\rangle=\langle\lambda_\nu,\psi^2\rangle$, 
		\begin{align*}
			\langle\lambda_\nu,\psi^2\rangle&=\langle\lambda_\nu,|\bD\bs u_\nu|^2\rangle\\
			&=\llangle\lambda_\nu\bD\bs u_\nu,\bD\bs u_\nu-\bD\bs u\rrangle+\llangle\lambda_\nu\bD\bs u_\nu,\bD\bs u\rrangle\\
			&=\int_{Q_T}\bs f\cdot(\bs u_\nu-\bs u)-\int_{Q_T}\partial_t\bs u_\nu\cdot(\bs u_\nu-\bs u)-\nu\int_{Q_T}\bD\bs u_\nu:\bD(\bs u_\nu-\bs u)\\
			&\hspace{0.5cm}+\int_{Q_T}\bs u_\nu\otimes\bs u_\nu:\nabla(\bs u_\nu-\bs u)+\llangle\lambda_\nu\bD\bs u_\nu,\bD\bs u\rrangle.
		\end{align*}
		Letting $\nu\rightarrow0$ and using \eqref{limiteLM}, we get
		$$\langle\lambda,\psi^2\rangle=\bs\llangle\bs\Xi,\bD\bs u\bs\rrangle$$
		concluding that
		$$\langle\lambda,|\bD\bs u|^2\rangle\le\langle\lambda,\psi^2\rangle=\bs\llangle\bs\Xi,\bD\bs u\bs\rrangle.$$
		On the other hand, recalling that $\lambda_\nu|\bD\bs u_\nu|^2=\lambda_\nu\psi^2$,
		\begin{align*}
			0\le\langle\lambda_\nu,|\bD(\bs u_\nu-\bs u)|^2\rangle&=\langle\lambda_\nu,|\bD\bs u_\nu|^2\rangle-2\llangle\lambda_\nu\bD\bs u_\nu,\bD\bs u\rrangle+\langle\lambda_\nu,|\bD\bs u|^2\rangle\\
			&\underset{\nu\rightarrow0}\longrightarrow\langle\lambda,\psi^2\rangle-2\bs\llangle\bs\Xi,\bD\bs u\bs\rrangle+\langle\lambda,|\bD\bs u|^2\rangle\\
			&=-\bs\llangle\bs\Xi,\bD\bs u\bs\rrangle+\langle\lambda,|\bD\bs u|^2\rangle=0,
		\end{align*}
		concluding that $\bs\llangle\bs\Xi,\bD\bs u\bs\rrangle=\langle\lambda,|\bD\bs u|^2\rangle$.
		Therefore, for any $\bs\xi\in\bs L^\infty(Q_T)$,
		$$\lim_{\nu\rightarrow0}\big|\langle\lambda_\nu,\bD(\bs u_\nu-\bs u)\cdot\bs\xi\rangle\big|\le\lim_{\nu\rightarrow0}\langle\lambda_\nu,|\bD(\bs u_\nu-\bs u)|\,|\bs\xi|\rangle\lim_{\nu\rightarrow0}\le\langle\lambda_\nu,|\bD(\bs u_\nu-\bs u)|^2\rangle^\frac12\langle\lambda_\nu,|\bs\xi|^2\rangle^\frac12=0.$$
		But then
		$$\bs\llangle\bs\Xi,\bs\xi\bs\rrangle=\lim_{\nu\rightarrow0}\bs\llangle\lambda_\nu\bD\bs u_\nu,\bs\xi\bs\rrangle=\lim_{\nu\rightarrow0}\langle\lambda_\nu,\bD\bs u_\nu\cdot\bs\xi\rangle=\lim_{\nu\rightarrow0}\langle\lambda_\nu,\bD\bs u\cdot\bs\xi\rangle=\langle\lambda,\bD\bs u\cdot\bs\xi\rangle\qquad\forall\bs\xi\in\bs L^\infty(Q_T),$$
		hence $\bs\Xi=\lambda\bD\bs u$ in $\bs L^\infty(Q_T)'$.
		
		As $\lambda_\nu|\bD\bs u_\nu|^2=\lambda_\nu\psi^2$, we have 
		$$\lambda\psi^2=\lim_{\nu\rightarrow0}\lambda_\nu|\bD\bs u_\nu|^2=\lim_{\nu\rightarrow0}\lambda_\nu|\bD\bs u|^2=\lambda|\bD\bs u|^2,$$
		concluding that $\lambda(|\bD\bs u|^2-\psi^2)=0$ and so, also $\lambda(|\bD\bs u|-\psi)=0$, since $|\bD\bs u|+\psi>0$.
		
		Finally in order to show that, for any $\nu\ge0$, we have $\bs u\in\mathscr K_{\psi}$, being a solution of \eqref{mlforte}, also solves the variational inequality \eqref{IV}, observe that for any  $\bs v\in\mathscr K_{\psi}$, we have
		\begin{equation}
			\label{ineq}
			\bs\llangle\lambda_\nu\bD\bs u,\bD(\bs v-\bs u)\bs\rrangle\le\langle\lambda_\nu,|\bD\bs u|(|\bD\bs v|-|\bD\bs u|)\rangle\le\psi^*\langle\lambda_\nu,\psi-|\bD\bs u|\rangle= 0,\end{equation}
		which implies that it solves the variational inequality globally in $Q_T$
		\begin{equation*}
			\int_{Q_T}\partial_t\bs u\cdot(\bs v-\bs u)+\nu\int_{Q_T}\bD\bs u:\bD(\bs v-\bs u)-\int_{Q_T}(\bs u\otimes\bs u):\nabla(\bs v-\bs u)\ge \int_{Q_T}\bs f\cdot(\bs v-\bs u).
		\end{equation*}

		Fixing $t\in(0,T)$ and $\delta>0$ such that $(t-\delta,t+\delta)\subset[0,T]$, define $\rho_\delta=\frac{\psi_*}{\psi_*+\eps_\delta}$, with $\displaystyle\eps_\delta=\sup\big\{\|\psi(t)-\psi(s)\|_{L^\infty(\Omega)}:t-\delta<s<t+\delta\big\}$. For any $\bs w\in\mathbb K_{\psi(t)}$ it is clear that $\bs v(s)=\rho_\delta \bs w$ when $|s-t|<\delta$ and $\bs v(s)=\bs u(s)$ otherwise, is such that $\bs v\in\mathscr K_{\psi}$ and we get, from the above inequality,
		\begin{multline*}
			\int_{t-\delta}^{t+\delta}\int_\Omega\partial_t\bs u\cdot(\bs w-\bs u)+\nu\int_{t-\delta}^{t+\delta}\int_\Omega\bD\bs u:\bD(\bs w-\bs u)\\
			-\int_{t-\delta}^{t+\delta}\int_\Omega(\bs u\otimes\bs u):\nabla(\bs w-\bs u)\ge \int_{t-\delta}^{t+\delta}\int_\Omega\bs f\cdot(\bs w-\bs u).
		\end{multline*}
		Then, dividing by $2\delta$ and letting $\delta\rightarrow0$, by Lebesgue theorem we obtain the conclusion.
	\end{proof}
	
	\begin{theorem}
		\label{tmlfraco}
		Assume \eqref{*}, \eqref{psiWeak} and \eqref{assumptions}. Then, for $\nu\geq0$, problem \eqref{mlfraco} has a solution $(\lambda,\bs u)\in L^\infty(Q_T)'\times  L^\infty\big(0,T;\mathbb V_p\big)$, for all $1<p<\infty$. In addition, $\bs u$ is a solution of the weak variational inequality \eqref{IVf}, which is unique in case $\nu>0$.
	\end{theorem}
	\begin{proof}
		Consider, for $\alpha\in(0,1)$, a generalized sequence of constraints $\psi_\alpha\in W^{1,\infty}\big(0,T;L^\infty(\Omega)\big)$ satisfying \eqref{*} and such that $\psi_\alpha\rightarrow\psi$ in $C\big([0,T];L^\infty(\Omega)\big)$ as $ \alpha\rightarrow 0$
		and let $(\lambda_\alpha,\bs u_\alpha)\in L^\infty(Q_T)'\times L^\infty\big(0,T;\mathbb V_p\big)\cap H^1\big(0,T;\bs L^2(\Omega)\big)$ be a sequence of generalized solutions of the strong Lagrange multipliers problem  \eqref{mlforte} with $\nu\ge0$.
		We rewrite this problem as follows
		\begin{multline}
			\label{apf}
			\int_\Omega\bs u_\alpha(T)\cdot\bs v(T)-\int_\Omega\bs u^0\cdot\bs v(0)	-\int_{Q_T}\bs u_\alpha\cdot\partial_t\bs v+\bs\llangle\lambda_\alpha D\bs u_\alpha,\bD\bs v\bs\rrangle\\
			+\nu\int_{Q_T}\bD\bs u_\alpha:\bD\bs v-\int_{Q_T}(\bs u_\alpha\otimes\bs u_\alpha):\nabla\bs v=\int_{Q_T}\bs f\cdot\bs v.
		\end{multline}
		We consider first the case $\nu>0$ and we recall that, from the proof of Theorem \ref{IVf}, the inequality \eqref{FFn} holds and so $\{\bs u_\alpha\}_\alpha$ is a Cauchy generalized sequence in $C\big([0,T];\bs L^2(\Omega)\big)\cap L^2\big(0,T;\mathbb V_2\big)$ for all  $\nu>0$.		
		
		Observe that, to obtain the equality \eqref{lambdaLambda} in the last theorem, we needed to use $\bs u$ as test function, which is not allowed here, since $\bs u$ may not have  a time derivative. But we can use $\bs u_\alpha$ as test function in \eqref{apf}. Write \eqref{apf} for $\bs u_\alpha$ with test function $\bs u_\alpha-\bs u_\beta$ and \eqref{apf} with $\alpha$ replaced by $\beta\in(0,1)$, also with test function  $\bs u_\alpha-\bs u_\beta$ and subtract the second equality to the first one. We get
		
		\begin{multline}
			\label{numu}
			\tfrac12\int_\Omega|\bs u_\alpha(T)-\bs u_\beta(T)|^2
			+\bs\llangle\lambda_\alpha \bD\bs u_\alpha -\lambda_\beta\bD \bs u_\beta,\bD(\bs u_\alpha-\bs u_\beta)\bs\rrangle\\
			+\nu\int_{Q_T}|\bD(\bs u_\alpha-\bs u_\beta)|^2
			-\int_{Q_T}(\bs u_\alpha\otimes\bs u_\alpha):\nabla\bs u_\beta+\int_{Q_T}(\bs u_\beta\otimes\bs u_\beta):\nabla\bs u_\alpha=0.
		\end{multline}

		As $\lambda_\alpha(|\bD\bs u_\alpha|-\psi)=0$, it is also holds $\lambda_\alpha(|\bD\bs u_\alpha|^2-\psi^2)=0.$
		The boundedeness of $\{\lambda_\alpha\}_\alpha$ in $L^\infty(Q_T)'$ implies the weak* convergence of a generalised subsequence to some $\lambda$, when $\alpha\rightarrow0$. Therefore,
		$$\lim_{\alpha}\bs\llangle\lambda_\alpha\bD\bs u_\alpha,\bD\bs u_\alpha\bs\rrangle=\lim_\alpha\langle\lambda_\alpha,\psi^2_\alpha\rangle=\langle\lambda,\psi^2\rangle.$$
		
		Analogously, the generalized sequence $\{\lambda_\alpha\bD\bs u_\alpha\}_\alpha$ is bounded in $\bs L^\infty(Q_T)'$ and so it has a subsequence converging weakly* to an element $\bs \Xi$, still labelled by $\alpha$. Recall that, for a subsequence, $\bD\bs u_\alpha\underset \alpha\lraup\bD\bs u$ and $\nabla\bs u_\alpha\underset \alpha\lraup\nabla\bs u$ in $\bs L^2(Q_T)$-weak.
		Taking the limit inf when $\alpha\rightarrow0$ in \eqref{numu}, we obtain
		\begin{multline}\label{mn}
			\tfrac12\int_\Omega|\bs u(T)-\bs u_\beta(T)|^2+\llangle\lambda,\psi^2\rrangle-\bs\llangle\bs\Xi,\bD\bs u_\beta\bs\rrangle-\bs\llangle\lambda_\beta\bD\bs u_\beta,\bD\bs u\bs\rrangle\\+\bs\llangle\lambda_\beta\bD\bs u_\beta,\bD\bs u_\beta\bs\rrangle
			+\nu\int_{Q_T}|\bD \bs u|^2-2\nu\int_{Q_T}\bD\bs u:\bD\bs u_\beta+\nu\int_{Q_T}|\bD\bs u_\beta|^2\\
			-\int_{Q_T}(\bs u\otimes\bs u):\nabla\bs u_\beta-\int_{Q_T}(\bs u_\beta\otimes\bs u_\beta):\nabla\bs u\le0.
		\end{multline}
		Taking now, in place of $\alpha$, the same subsequence of $\beta$, and letting $\beta\rightarrow0$, as
		$$\lim_\beta\llangle\lambda_\beta,|\bD\bs u_\beta|^2\rrangle=\lim_\beta\langle\lambda_\beta,\psi^2_\beta\rangle=\bs\langle\lambda,\psi^2\rangle,$$
		from \eqref{mn} we get
		$$2\langle\lambda,\psi^2\rangle-2\bs\llangle\bs\Xi,\bD\bs u\bs\rrangle\le0.$$

		On the other hand, as $|\bD\bs u|\le\psi$ and $\lambda\ge0$,
		\begin{align*}
			0\le\lim_\alpha \langle\lambda_\alpha,|\bD(\bs u_\alpha-\bs u)|^2\rangle&=\lim_\alpha\Big(\langle\lambda_\alpha,|\bD\bs u_\alpha|^2\rangle-2\bs\llangle\lambda_\alpha\bD\bs u_\alpha,\bD\bs u\bs\rrangle+\langle\lambda_\alpha,|\bD\bs u|^2\rangle\Big)\\
			&=\langle\lambda,\psi^2\rangle-2\bs\llangle\bs\Xi,\bD\bs u\bs\rrangle+\langle\lambda,|\bD\bs u|^2\rangle\\
			&\le 2\langle\lambda,\psi^2\rangle-2\bs\llangle\bs\Xi,\bD\bs u\bs\rrangle\le0.
		\end{align*}
		
		Finally, for any $\bs\xi\in\bs L^\infty(Q_T)$, we have
		\begin{align*}
			\lim_\alpha	\big|\bs\llangle\lambda_\alpha\bD(\bs u_\alpha-\bs u),\bs\xi\bs\rrangle\big|&\le\lim_\alpha\langle\lambda_\alpha,|\bD(\bs u_\alpha-\bs u)||\bs\xi|\rangle\le\lim_\alpha\langle\lambda_\alpha,|\bD(\bs u_\alpha-\bs u)|^2\rangle^\frac12\langle\lambda_\alpha,|\bs \xi|^2\rangle^\frac12\\
			&\le\lim_\alpha\langle\lambda_\alpha,|\bD(\bs u_\alpha-\bs u)|^2\rangle^\frac12\|\lambda_\alpha\|_{L^\infty(Q_T)}^\frac12\|\bs\xi\|_{\bs L^\infty(Q_T)}^\frac12=0,
		\end{align*}
		concluding that $\bs\Xi=\lambda\bD\bs u$ in $\bs L^\infty(Q_T)'$, being the rest of the proof identical to the proof in the strong case.
		
		We consider now the degenerate case $\nu=0$
		
		We consider a family of functions $\psi_\nu\in W^{1,\infty}\big(0,T;L^\infty(\Omega)\big)$ such that $\psi_\nu\rightarrow\psi$ in $C\big([0,T];L^\infty(\Omega)\big)$ when $\nu\rightarrow0$ and a pair $(\lambda_\nu,u_\nu)$ that solves the strong Lagrange multiplier problem \eqref{mlforte}.
		Note that, by the estimates obtained above, we know that $\{\bD\bs u_\nu\}_\nu$, $\{\lambda_\nu\}_\nu$ and $\{\lambda_\nu \bD\bs u_\nu\}_\nu$ are bounded independently of $\nu$, respectively, in $\bs L^\infty(Q_T)$, $L^\infty(Q_T)'$ and $\bs L^\infty(Q_T)'$, for $0<\nu<1$. So, we can extract a generalised subsequence, still denoted by $\nu$, with $\nu\rightarrow0$, that satisfies 
		$\bs u_\nu\lraup\bs u$ in $L^\infty\big(0,T;\mathbb V_p\big)$-weak* for all  $1<p<\infty$, as well as $\lambda_\nu\lraup\lambda$ in $L^\infty(Q_T)'$-weak* and $\lambda_\nu\bD\bs u_\nu\lraup\bs\Xi$ in $\bs L^\infty(Q_T)'$-weak*. We need to show that $\bs\Xi=\lambda\bD\bs u$ in order to conclude that $(\lambda,\bs u)\in L^\infty(Q_T)'\times  \mathscr V_\infty^\cap$ solves the problem \eqref{mlfraco} in the limit case $\nu=0$. Arguing as above, it is sufficient to show that $\bs u_\nu\longrightarrow\bs u \text{ in }C\big([0,T];\bs L^2(\Omega)\big)$. This can be easily done by proving that the sequence $\{\bs u_\nu\}_\nu$ is relatively compact in $C\big([0,T];\bs L^p(\Omega)\big)$, $1< p<\infty$, which follows essentially the same steps as in the proof of Theorem \ref{Ew} for the case $\nu=0$ by using in \eqref{neta} $\nu$ instead $\nu_n$.

		Finally using  \eqref{ineq} we can show, in both cases $\nu>0$ and $\nu=0$, that $\bs u$ solves the weak variational inequality similarly to the case of strong solutions. In the case $\nu>0$ we know it has a unique solution.
	\end{proof}
	
	\section{Solutions of the quasi-variational inequality with $\nu>0$}

	We consider now the case where the constraint $\psi(x,t)=\Psi[\bs u](x,t)>0$ depends on the solution $\bs u$ through a functional $\Psi: \mathscr X=D(\Psi)\subset\bs L^2(Q_T)\rightarrow \R^+$. In this section we restrict to the case $\nu>0$.
	
	We say that the pair $(\lambda, \bs u)$, belonging to $ L^\infty(Q_T)'\times  \mathscr V_2$, is a weak solution of the implicit Lagrange multiplier problem if
	\begin{equation}\label{qvimlfraco}
		\begin{cases}
			\displaystyle-\int_{Q_T}\bs u\cdot\partial_t\bs v+\bs\llangle\lambda\bD\bs u,\bD\bs v\bs\rrangle+\nu\int_{Q_T}\bD\bs u:\bD\bs v
			-\int_{Q_T}(\bs u\otimes\bs u):\nabla\bs v\\
			\hfill{\displaystyle =\int_{Q_T}\bs f\cdot\bs v+\int_\Omega\bs u^0\cdot\bs v(0),\quad \forall\bs v\in \mathscr V_\infty\cap\mathscr W_2,\  \bs v(T)=\bs 0,}\\
			|\bD\bs u|\le \Psi[\bs u]\ \text{ in }Q_T,\quad\lambda\ge0\ \text{ and }\ \lambda\big(|\bD\bs u|-\Psi[\bs u]\big)=0\ \text{ in }L^\infty(Q_T)'.
		\end{cases}
	\end{equation}
	We observe that, as in the variational case, when the constraint is bounded, $\bs u\in L^\infty\big(0,T;\mathbb V_p\big)$, for all  $1<p<\infty$, is the weak solution of the following quasi-variational inequality
	\begin{equation} \label{qIVf}
		\begin{cases}
			\bs u\in\mathscr K_{\Psi[\bs u]},\\
			\displaystyle	\int_0^T\langle\partial_{t}\bs v,\bs v-\bs u\rangle
			+\nu\int_{Q_T}\bD\bs u:\bD(\bs v-\bs u)
			-\int_{Q_T}(\bs u\otimes\bs u):\nabla(\bs v-\bs u)\\
			\hfill{\displaystyle	\ge\int_{Q_T}\bs f\cdot(\bs v-\bs u)-\tfrac12\int_\Omega |\bs v(0)-\bs u^0|^2,\quad \forall\,\bs v\in\mathscr K_{\Psi[\bs u]}\cap\mathscr W_2,}
		\end{cases}
	\end{equation}
	where $\mathscr K_{\Psi[\bs u]}$ is defined in \eqref{ConvexSetWeak} with $\psi$ replaced by $\Psi[\bs u]$.

	Let the domain $\mathscr X=D(\Psi)$ of $\Psi$ be a vector space such that
	\begin{equation}\label{Psi}
		\mathscr V_2\subset\mathscr X\subset\bs L^2(Q_T) \quad \text{ with continuous inclusions}.  
	\end{equation}
	
	We assume also that there exist positive constants $\psi_*$ and $\psi^*$ such that
	\begin{equation}\label{psi_*^*}
		0<\psi_*\le \Psi[\bs u](x,t)\le \psi^*,\  \forall \bs u\in\mathscr X\text{ for a.e. }(x,t)\in Q_T.
	\end{equation}

	We give first existence results using a compactness argument under the general assumption
	\begin{equation}\label{Psif}
		\Psi:\mathscr X\rightarrow C\big([0,T];\bs L^\infty(\Omega)\big)\quad\text{ is completely continuous.}
	\end{equation}
	
	\begin{theorem}
		\label{schauder_1}
		Let $\nu>0$, $\bs f\in\bs L^2(Q_T)$, $\bs u^0\in\K_{\Psi[\bs u^0]}$ and $\Psi$ satisfy \eqref{Psi} and \eqref{psi_*^*}. If, in addition,
		$\Psi$ satisfies \eqref{Psif} , then the problem \eqref{qvimlfraco} has a solution $(\lambda,\bs u)\in L^\infty(Q_T)'\times \big(L^\infty\big(0,T;\mathbb V_p\big)\cap C\big([0,T];L^p(\Omega)\big)\big)$, for all $1<p<\infty$.
	\end{theorem}
	\begin{proof}	Consider firstly the case $\mathscr X=\bs L^2(Q_T)$.
		
		Let $S:\bs L^2(Q_T)\rightarrow\bs L^2(Q_T)$ be defined by $S(\bs v)=S(\bs f,\bs u^0,\Psi[\bs v])$, the unique solution  of the weak variational inequality \eqref{IVf} with data $\bs f,\bs u^0,\psi=\Psi[\bs v]$.
		
		The operator $S$ is continuous.	Indeed,
		if $\bs v_n\underset{n}\rightarrow \bs v$ in $\bs L^2(Q_T)$, using the estimate \eqref{Ff}, we conclude that
		\begin{equation*}
			\|S(\bs v_n)-S(\bs v)\|_{_2}^2=\|\bs u_n-\bs u\|_{_2}^2\le T\,e^{C_\nu T}C_*\|\Psi[\bs v_n]-\Psi[\bs v]\|_{_\infty},
		\end{equation*}
		with $C_*, C_\nu$ defined in \eqref{constants}. Thus, the complete continuity of $\Psi$ implies the continuity and the compactness of $S$.
		
		By the estimate \eqref{apriori} we obtain
		\begin{equation*}
			\|S(\bs v)\|_{_2}\le \sqrt Te^\frac{T}2\big(\|\bs f\|_{_2}^2+\|\bs u^0\|_{_2}^2\big)^\frac12=R
		\end{equation*}
		concluding that $S(B_R)\subset B_R$, where $B_R$ denotes the ball  centered at $\bs 0$ with radius $R$ of the vector space $\mathscr X=\bs L^2(Q_T)$. Therefore, by the Schauder's  fixed point theorem, $S$ has a fixed point $\bs u=S(\bs u)$, which is a solution of the quasi-variational inequality \eqref{qIVf}. 
		
		We observe that, denoting by $(\lambda(\bs v),S(\bs v))$ a solution of the Lagrange multiplier problem \eqref{mlfraco} with data $\bs f, \bs u^0,\Psi[\bs v]$ then, if $\bs u$ is a fixed point  of $S$, the pair $(\lambda(\bs u),S(\bs u))$ solves the Lagrange multiplier problem \eqref{qvimlfraco}. 
		
		In the case $\mathscr V_2\subset\mathscr X \subsetneqq \bs L^2(Q_T)$, we proceed analogously. Using 
		the {\em a priori} estimate \eqref{apriori} in the estimate \eqref{apriori_extra} for $t=T$, we obtain first
		\begin{equation} \label{R1}
			\|\bs u\|_{\mathscr V_2}=\|\bD\bs u\|_{_2}\le \left( \Bigg(\frac{1+Te^T}{2\nu}\Bigg)\Big(\|\bs f\|_{_2}^2+\|\bs u^0\|_{_2}^2\Big)\right)^\frac12 =R_1.
		\end{equation}
		Therefore $S_1(B_{R_1})\subset B_{R_1}$, where now $B_{R_1}$ denotes the ball  centered at $\bs 0$ with radius $R_1$ of the vector space $\mathscr V_2$, where we apply the Schauder's  fixed point theorem to the operator $S_1:B_{R_1}\rightarrow B_{R_1}$ that assigns to each $\bs v$ the unique solution $\bs u=S_1(\bs f,\bs u^0,\Psi[\bs v])$. We easily conclude that the operator $S_1$ is also completely continuous, by using the estimate \eqref{FF} in the form
		$$\|S_1(\bs v_n)-S_1(\bs v)\|^2_{_\mathscr X}\le C_{_\mathscr X}\|\bs u_n-\bs u\|^2_{\mathscr V_2}\le C_{_\mathscr X}\frac{C_*}{\nu}\big(1+TC_\nu e^{C_\nu T}\big)\|\Psi[\bs v_n]-\Psi[\bs v]\|_{_\infty}$$
		and the assumption \eqref{Psif}.
	\end{proof}
	
	Now we present two relevant examples of constraints $\Psi$ with the sufficient properties to solve the quasi-variational inequalities under consideration.
	
	\begin{example} We start with a constraint of the type   
		$$\Psi[\bs v]=\phi(x,t,\bs\zeta(\bs v)(x,t))\qquad\text{ a.e. in }Q_T,$$
		where
		$\phi=\phi(x,t,\bs \zeta):Q_T\times\R^d\rightarrow\R$ is a function satisfying 
		$$0<\psi_*\le \phi(x,t,\bs\zeta)\le\psi^*,\quad\text{ for a.e. }(x,t)\in Q_T,$$
		uniformly continuous in the variables $x, t$ and $\bs\zeta$.

		We consider the nonlocal integral operator 
		\begin{equation*} 
			\bs\zeta(\bs v)(x,t)=\int_0^t\int_{\Omega}\bs v(y,s)K(x,t,y,s)dyds,\quad\text{ for a.e. }(x,t)\in Q_T,
		\end{equation*}
		where $K=K(x,t,y,s)$ is a kernel $K\in C\big(\overline Q_T\times\overline Q_T\big)$. Since $\bs\zeta:\bs L^2(Q_T)\rightarrow C\big([0,T];L^\infty(\Omega)\big)$ is a compact operator (see, for instance, Corollary 5.1 of \cite{Eveson1995}), it follows that $\Psi:\bs L^2(Q_T)\rightarrow C\big([0,T];L^\infty(\Omega)\big)$ satisfies the assumptions of Theorem \ref{schauder_1}. Therefore we obtain a weak solution to the quasi-variational inequality \eqref{qIVf}, as well as the solvability of the implicit Lagrange multiplier problem \eqref{qvimlfraco}.
		
	\end{example}
	
	\begin{example} Another example is given when the functional $\Psi$ is defined through the coupling with a linear parabolic equation with convection, for instance when the threshold of the thick fluid depends on the temperature or on a chemical composition.
		We may define a compact operator through the unique solution $\zeta$, in a bounded open domain $\Omega$, with boundary $\partial\Omega$ of class $C^2$, of the Cauchy-Dirichlet in $Q_T$ problem for the convective parabolic scalar equation
		\begin{align}\label{p12}
			\begin{cases}
				\partial_t\zeta-\nabla\cdot \big(\bs A(x,t)\nabla\zeta\big)+\bs v\cdot \nabla\zeta=\varrho(x,t)\quad\text{in }Q_T,\\
				\zeta=0\quad\text{on }\partial\Omega\times(0,T),\quad\zeta(0)=\zeta_0\quad\text{on }\Omega,
			\end{cases}
		\end{align}
		where the strictly elliptic matrix $\bs A$ has bounded continuous coefficients with bounded  $\nabla\cdot\bs A$.
		
		If $\zeta_0\in W_0^{2-\frac{2}{q}, q}(\Omega)$ and $\varrho\in L^q(Q_T)$, with $q>d+2$, the following Solonnikov estimate holds (see \cite[p. 341]{LSU1968})
		\begin{equation}\label{sol-q}
			\|\zeta\|_{W^{2,1;q} ( Q_T)}\leq C\Big(\|\zeta_0\|_{W^{2-\frac2q,_q}_0(\Omega)}+ \|\varrho\|_{L^q(Q_T)}\Big),
		\end{equation}
		being $C>0$ a constant independent  $\varrho$ but depending on 
		$\bs v\in \bs L^q(Q_T)$. Here $W^{2,1;q}(Q_T)$ denotes the Sobolev space of the $L^q(Q_T)$ functions $\zeta$ such that $\partial_t\zeta$, $\nabla\zeta$ and the Hessian $\nabla^2\zeta$ also belong to $L^q(Q_T)$. Since $q>d+2$, by well-known embeddings,  $\zeta\in W^{2,1;q}(Q_T)\subset C^{1,0;\beta}(\overline Q_T)$, so that $\zeta$ and  $\nabla\zeta\in C^{0,\beta}(\overline Q_T)$ are $\beta$-Hölder continuous in $(x,t)\in \overline {Q}_T$, being that inclusion also compact for some $0<\beta<1$. It follows that the linear solution map $L^q(Q_T)\ni \bs v \mapsto(\zeta, \nabla \zeta)\in C^{0,\beta}(\overline Q_T)$ is completely continuous, by the estimate \eqref{sol-q}, for $q>d+2$, and the compactness of the inclusion of $\zeta\in W^{2,1;q}(Q_T)\subset C^{1,0;\beta}(\overline Q_T)$.

		Let $\psi=\psi(x,t,z,\bs\xi)$ be uniformly continuous on the variables $t,z,\bs\xi$, satisfying $0<\psi_*\le\psi\le\psi^*$, and set 
		$$\Psi[\bs v]=\psi\big(x,t,\zeta_{\bs v}(x,t),\nabla\zeta_{\bs v}(x,t)\big),$$	
		where  $\zeta_{\bs v}$ is the solution of \eqref{p12}.
		
		Therefore the mapping $\Psi:\bs L^q(Q_T)\cap \mathscr V_2\rightarrow C\big([0,T];\bs L^\infty(\Omega)\big)$ is also completely continuous. As in Theorem \ref{schauder_1}, we obtain a solution to the implicit Lagrange multiplier problem \eqref{qvimlfraco} and to the weak quasi-variational inequality \eqref{qIVf} coupled with the Cauchy-Dirichlet problem \eqref{p12}.
	\end{example} 
	
	Analogously, a pair  $(\lambda,\bs u)$, belonging to  $ L^\infty(Q_T)'\times\big(\mathscr V_2\cap H^1\big(0,T;\bs L^2(\Omega)\big)\big)$, is a strong solution of the implicit Lagrange multiplier problem if it satisfies
	\begin{equation}\label{qvimlforte}
		\begin{cases}
			\displaystyle\int_{Q_T}\partial_t\bs u\cdot\bs v+\bs\llangle\lambda\bD\bs u,\bD\bs v\bs\rrangle+\nu\int_{Q_T}\bD\bs u:\bD\bs v
			-\int_{Q_T}(\bs u\otimes\bs u):\nabla\bs v\\
			\displaystyle  \qquad\qquad=\int_{Q_T}\bs f\cdot\bs v,\quad \forall\bs v\in \mathscr V_\infty,\\
			\bs u(0)=\bs u^0,\\
			|\bD\bs u|\le \Psi[\bs u]\ \text{ in }Q_T,\quad\lambda\ge0\ \text{ and }\ \lambda\big(|\bD\bs u|-\Psi[\bs u]\big)=0\ \text{ in }L^\infty(Q_T)'.
		\end{cases}
	\end{equation}

	Similarly $\bs u\in H^1\big(0,T;\bs L^2(\Omega)\big)\cap  L^\infty\big(0,T;\mathbb V_p\big)$, for all  $1<p<\infty$, is also the strong solution of the quasi-variational inequality
	\begin{equation} \label{qIV}
		\begin{cases}
			\bs u =\bs u(t)\in\K_{\Psi[u](t)}\ \text{for}\ t\in(0,T),\ \bs u(0)=\bs u^0,\\
			\displaystyle\int_{\Omega}\partial_{t}\bs u\cdot(\bs w-\bs u)(t)
			+\nu\int_{\Omega}\bD\bs u:\bD(\bs w-\bs u)(t)
			-\displaystyle\int_{\Omega}(\bs u\otimes\bs u):\nabla(\bs w-\bs u)(t)\\
			\displaystyle	\qquad\qquad\ge\int_{\Omega}\bs f\cdot(\bs w-\bs u)(t),\quad\forall\,\bs w\in\K_{\Psi[u](t)},\ \text{a.e.}\ t\in(0,T),
		\end{cases}
	\end{equation}
	with $\K_{\Psi[\bs u](t)}$ defined in \eqref{ConvexSet-t}, replacing $\psi$ by $\Psi[\bs u]$.

	We consider now sufficient conditions to solve the implicit problem by imposing a contractive behaviour to the solution operator of the variational problem. Let now the nonlocal operator $\Psi: \mathscr X=\mathscr V_2\rightarrow\R^+$ be defined, for each $(x,t)\in Q_T$, as in \cite{HR2017}, by 
	\begin{equation}\label{PG}
		\Psi[\bs w](x,t)=\Gamma(\bs w)\varphi(x,t),\end{equation}
	where $\Gamma: \mathscr V_2\rightarrow\R$ is a locally Lipchitz continuous operator satisfying
	\begin{equation}\label{gamma}  
		\begin{cases}
			&\!\! 0<\eta(R)\le\Gamma(\bs w)\le \rho(R)\quad \forall\,\bs w\in B_R,\\
			&\!\!|\Gamma(\bs w_1)-\Gamma(\bs w_2)|\le\gamma(R)\|w_1-w_2\|_{\mathscr V_2}\quad\ \forall\,\bs  w_1,\bs  w_2\in  B_{R},
		\end{cases}
	\end{equation}
	for $R>0$. Here the real functions $\eta,\gamma$ and $\rho$ are strictly positive in $\R$ and $B_R=\{\bs w\in \mathscr V_2: \|\bs w\|_{\mathscr V_2}\le R\}$. The function $\varphi: Q_T\rightarrow\R^+$ satisfies the assumption
	\begin{equation} \label{phi*}
		0<\varphi_*\leq\varphi(x,t)\leq \varphi^* \quad\text{a.e. }\quad (x,t)\in Q_T.
	\end{equation}
	
	Assuming $\bs f\in\bs L^2(Q_T)$ and $\bs u^0\in\bs L^2(\Omega)$, if we set $\bs w=0$ in \eqref{qIV} and, as in the proof of the previous theorem, we use 
	the estimate \eqref{apriori} in the estimate \eqref{apriori_extra} for $t=T$, we obtain also the {\em a priori} estimate \eqref{R1}. Hence we conclude also that any quasi-variational solution $\bs u\in B_1$. Similarly, also choosing zero as test function, we obtain the estimate
	\begin{align*}
		\nu \int_{Q_T}|\bD\bs u|^2&\le\int_{Q_T}\bs f\cdot\bs u+\tfrac12\int_\Omega|\bs u^0|^2\\
		&\le\tfrac\nu{2}\int_{Q_T}|\bD\bs u|^2+\tfrac1{\nu\kappa}\int_{Q_T}|\bs f|^2+\tfrac12\int_\Omega|\bs u^0|^2, 
	\end{align*}
	where $\kappa$ is given in \eqref{apriori2}, and any quasi-variational solution also satisfies the {\em a priori} estimate 
	\begin{equation}\label{novoR*}
		\|\bs u\|_{\mathscr V_2}\le \big(\tfrac2{\nu^2\kappa}\|\bs f\|^2_{_2}+\tfrac1{\nu}\|\bs u^0\|^2_{_2}\big)^\frac12=R_2,
	\end{equation}
	and therefore we have {\em a priori} $\bs u\in B_{R_*}$, since 
	\begin{equation}\label{R*} 
		\|\bs u\|_{\mathscr V_2}\le R_*=\min \{R_1,R_2\}, 
	\end{equation}
	where $R_1$ is defined in $\eqref{R1}$.

	\begin{theorem} \label{thm5.2} Let $\bs f\in\bs L^2(Q_T)$, $\bs u^0\in\mathscr K_{\Psi[\bs u^0]}$ and let $\Psi[\bs u]$ be the functional defined by \eqref{PG} under the assumptions \eqref{gamma} and \eqref{phi*}. If  $\varphi\in C\big([0,T];L^\infty(\Omega)\big)$ (respectively $\varphi\in W^{1,\infty}\big(0,T;L^\infty(\Omega)\big)$), and
		\begin{equation}
			\label{contra}
			\tfrac{\gamma_*}{\eta_*\sqrt \nu}\Big[\sqrt \nu R_*+\big[\big(e^{BT}-1\big)\big(\| A+\|\bs u^0\|_{_2}^2\big)+A\big]^\frac12 \Big] <1,
		\end{equation}
		with $\gamma_*=\gamma(R_*)$ and $\eta_*=\eta(R_*)$, where
		\begin{equation}\label{ab1}
			A=\|\bs f\|_{_2}^2+4d^2M^4_*\,
			\tfrac{\rho_*^2}{\nu\eta_*^2}T|\Omega|
		\end{equation}
		\begin{equation}\label{ab2}
			B=1+\tfrac{4d}{\nu}M^2_*
		\end{equation}

		\noindent and $M_*$ is defined as in \eqref{infty}, with $\psi^*=\rho_*\varphi^*$, $\rho^*=\rho(R_*)$,
		then the implicit  Lagrange multiplier problem \eqref{qvimlfraco} (resp. \eqref{qvimlforte}) has a weak solution $(\lambda,\bs u)\in  L^\infty(Q_T)'\times\big(L^\infty(0,T;\mathscr V_p)\cap C\big([0,T];\bs L^p(\Omega)\big)$, for all $1<p<\infty$ (resp. a strong solution $(\lambda,\bs u)\in L^\infty(Q_T)'\times(L^\infty(0,T;\mathscr V_p)\cap H^1\big(0,T;\bs L^2(\Omega)\big)\cap C\big([0,T];\bs L^p(\Omega)\big)$, for all $1<p<\infty$), being  $\bs u$ the unique weak (resp. strong) solution of the quasi-variational inequality \eqref{qIVf} (resp. \eqref{qIV}). 
	\end{theorem}
	\begin{proof}
		In this proof we extend the argument of \cite{RodriguesSantos2019} by estimating the additional nonlinear convection term and imposing sufficient conditions for a solution map $S$ to be a contraction, so that the Banach fixed point theorem applies.
		
		Suppose first $\varphi\in W^{1,\infty}\big(0,T;L^\infty(\Omega)\big)$ and define $S:B_{R_*}\rightarrow \mathscr V_2$ by $S(\bs v)=S(\bs f,\bs u^0,\Psi[\bs v])=\bs u$, being $\bs u$ the unique strong solution of the variational inequality \eqref{IV} with data $\bs f$, $\bs u^0$ and $\psi=\Psi[\bs v]$.  
		Note that $S(B_{R_*})\subset B_{R_*}$, by estimate \eqref{R*}.
		
		For $i=1,2$, let $\bs v_i\in B_{R_*}$ and $\bs u_i=S(\bs f,\bs u^0,\Psi[\bs v_i])$.  Set $\mu=\frac{\Gamma(\bs v_2)}{\Gamma(\bs v_1)}$, assuming without loss of generality that $\mu>1$. Denoting $\psi=\Psi[\bs v_1]=\Gamma(\bs v_1)\varphi$, we have $\Psi[\bs v_2]=\Gamma(\bs v_2)\varphi=\mu \psi$.  
		
		Noting that
		\begin{equation}\label{dividir}
			\|S(\bs v_1)-S(\bs v_2)\|_{\mathscr V_2}=\|\bD(\bs u_1-\bs u_2)\|_2\le\|\bD(\bs u_1-\mu\bs u_1)\|_2+\|\bD(\mu\bs u_1-\bs u_2)\|_2,
		\end{equation}
		the first term is easily estimated by
		\begin{equation}\label{u1}
			\|\bD(\bs u_1-\mu\bs u_1)\|_{_2}=(\mu-1)\|\bD\bs u_1\|_{_2}\le R_*\frac{\gamma_*}{\eta_*}\|\bD(\bs v_1-\bs v_2)\|_{_2}
		\end{equation}
		since we have, by \eqref{gamma} with $\gamma_*=\gamma(R_*)$ and $\eta_*=\eta(R_*)$,
		\begin{equation} \label{mu-1}
			\mu-1=\frac{\Gamma(\bs v_2)- \Gamma(\bs v_1)}{\Gamma(\bs v_1)}\le\frac{\gamma_*}{\eta_*}\|\bD(\bs v_1-\bs v_2)\|_{_2}.
		\end{equation}
		
		In order to estimate the second term of the right hand side of \eqref{dividir}, we multiply both sides of the variational inequality \eqref{IV} for $\bs u_1$ by $\mu^2$, and, as $\bs u_2(t)\in\K_{\mu \psi(t)}$, we may choose $\bs w(t)=\frac1\mu\bs u_2(t)\in\K_{\psi(t)}$, for a.e. $t\in (0,T),$ to get
		\begin{multline*}
			\int_\Omega\partial_t(\mu\bs u_1)\cdot(\mu\bs u_1-\bs u_2)(t)+\nu\int_\Omega \bD(\mu \bs u_1):\bD(\mu\bs u_1-\bs u_2)(t)\\
			-\int_\Omega(\mu\bs u_1\otimes\mu\bs u_1:\nabla(\mu\bs u_1-\bs u_2)(t)
			\le\int_\Omega\mu\bs f\cdot(\mu\bs u_1-\bs u_2)(t)\\
			+(\mu^2-\mu) \int_\Omega(\bs u_1\otimes\bs u_1):\nabla(\mu\bs u_1-\bs u_2)(t).
		\end{multline*}
		As $\bs u_2=S(\bs f,\bs u^0,\mu \psi)$ and $\mu\bs u_1(t)\in\K_{\mu\psi(t)}$, we also have
		\begin{multline*}
			\int_\Omega\partial_t\bs u_2\cdot(\mu\bs u_1-\bs u_2)(t)+\nu\int_\Omega \bD( \bs u_2):\bD(\mu\bs u_1-\bs u_2)(t)\\
			-\int_\Omega(\bs u_2\otimes\bs u_2):\nabla(\mu\bs u_1-\bs u_2)(t)\ge \int_\Omega\bs f\cdot(\mu\bs u_1-\bs u_2)(t).
		\end{multline*}
		
		From these two inequalities, for $\bs \zeta=\mu\bs u_1-\bs u_2$, we obtain   
		\begin{align}\label{aux}
			\tfrac12\tfrac{d\ }{dt}	\int_\Omega|\bs \zeta(t)|^2+\nu\int_\Omega|\bD\bs \zeta(t)|^2
			&\le(\mu-1)\int_\Omega\bs f(t)\cdot\bs \zeta(t)+\Xi(t)+\Upsilon(t)\\
			\nonumber &\le \tfrac12\int_\Omega |\bs \zeta(t)|^2+\tfrac12(\mu-1)^2\int_\Omega|\bs f(t)|^2+\Xi(t)+\Upsilon(t)\\
			\nonumber      &\le\tfrac\nu2\int_\Omega|\bD\bs\zeta(t)|^2+(\tfrac12+\tfrac{2d}\nu M_*^2)\int_\Omega|\bs\zeta(t)|^2+(\mu-1)^2A_*(t),
		\end{align}
		with $A_*(t)=\frac12\|\bs f(t)\|^2_{_2}+\frac{2\rho_*}{\nu\eta_*^2}M^4_*|\Omega|d^2$, since
		\begin{align*}
			\Xi(t)&=\mu(\mu-1)\int_\Omega(\bs u_1\otimes\bs u_1):\nabla\bs \zeta(t)\le(\mu-1)\tfrac{\rho_*}{\eta_*}\|(\bs u_1\otimes\bs u_1)(t)\|_{_2}\sqrt 2\|\bD\bs\zeta(t)\|_{_2}\\
			&\le\tfrac\nu4\|\bD\bs\zeta(t)\|_{_2}^2+\tfrac{(\mu-1)^2}\nu\,\tfrac{2\rho^2_*}{\eta_*^2}M^4_*|\Omega|d^2
		\end{align*}
		and
		\begin{align*}
			\Upsilon(t)&=\int_\Omega\big((\mu\bs u_1\otimes\mu\bs u_1)-(\bs u_2\otimes\bs u_2):\nabla\bs\zeta\big)(t)=\int_\Omega(\bs u_2\otimes\bs\zeta):\nabla\bs\zeta(t)\\
			&\le\sqrt2\|(\bs u_2\otimes\bs\zeta)(t)\|_{_2}\|\bD\bs\zeta(t)\|_{_2}\le\sqrt{2d}M_*\|\bs\zeta(t)\|_{_2}\|\bD\bs\zeta(t)\|_{_2}\\
			&\le\tfrac\nu4\|\bD\bs\zeta(t)\|_{_2}^2+\tfrac{2d}\nu M_*^2\|\bs\zeta(t)\|_{_2}^2.
		\end{align*}
		
		Noting that $\bs\zeta(0)=(\mu-1)\bs u^0$, setting $ y(t)=\|\zeta(t)\|_{_2}^2$ and integrating in $t$, from \eqref{aux} we obtain
		$$y(t)\le\big(\mu-1)^2(A+\|\bs u^0\|_{_2}^2\big)+B\int_0^t y(s)ds,$$
		where $A$ is given by \eqref{ab1} and $B$ by \eqref{ab2}. By Gr\"onwall integral inequality we have
		$$\int_0^ty(s)ds\le (\mu-1)^2\big(A+\|\bs u^0\|_{_2}^2)\big(e^{Bt}-1\big)/B.$$
		
		Integrating \eqref{aux} between $0$ and $T$ and using this estimate, we deduce
		$$\|\bD\bs\zeta\|_{_2}^2\le\tfrac{(\mu-1)^2}\nu\Big[\big(A+\|\bs u^0\|_{_2}^2)\big(e^{BT}-1\big)+A\Big]\le\tfrac{\gamma_*^2}{\nu\eta_*^2}\Big[\big(A+\|\bs u^0\|_{_2}^2)\big(e^{BT}-1\big)+A\Big]\|\bs v_1-\bs v_2\|^2_{\mathscr V_2}.$$
		From \eqref{dividir} and \eqref{u1}, this yields the contraction
		$$\|S(\bs v_1)-S(\bs v_2\|_{\mathscr V_2}\le\tfrac{\gamma_*}{\eta_*\sqrt\nu}\left[\sqrt\nu\,R_*+\Big[\big(A+\|\bs u^0\|_{_2}^2)\big(e^{BT}-1\big)+A\Big]^\frac12\right]\|\bs v_1-\bs v_2\|_{\mathscr V_2},$$
		provided \eqref{contra} holds.
	\end{proof}

	\begin{example} 
		
		Let $\Psi:\mathscr V_2\rightarrow\R$ be defined by
		$$\Psi[\bs w](x,t)=\varphi(x,t)\big(\alpha+\beta\int_{Q_T}|\bD\bs w|^2\big),\quad (x,t)\in Q_T,$$
		where $\alpha,\beta>0$.
		
		We can apply directly Theorem \ref{thm5.2}, by observing that in $B_{R_*}\subset \mathscr V_2$
		$$0<\eta(R_*)=\alpha\le \Gamma(\bs w)=\alpha+\beta\int_{Q_T}|\bD\bs w|^2\le\alpha+\beta R_*^2=\rho(R_*)$$
		and
		\begin{align*}
			|\Gamma(\bs w_1)-\Gamma(\bs w_2)|&=\beta\Big|\int_{Q_T}|\bD\bs w_1|^2-|\bD\bs w_2|^2\Big|=\beta\Big|\int_{Q_T}(\bD\bs w_1-\bD\bs w_2)\cdot(\bD\bs w_1+\bD\bs w_2)\Big|\\
			&\le 2\beta R_*\|\bs w_1-\bs w_2\|_{\mathscr V_2}= \gamma(R_*)\|\bs w_1-\bs w_2\|_{\mathscr V_2}.
		\end{align*}
		
	\end{example}
	
	\begin{remark}
		Under the assumptions of Theorem \ref{thm5.2}, if we take any $\bs v_0 \in B_{R_*}$ and we consider $\bs v_n=S(\bs v_{n-1})$, for $n\in \N$, the sequence $\{\bs v_n\}$ converges strongly in $\mathscr V_2$ to the fixed point $\bs u=S(\bs u)$, which is the unique weak (resp. strong) quasi-variational solution corresponding to $\varphi\in C\big([0,T];L^\infty(\Omega)\big)$ (respectively $\varphi\in W^{1,\infty}\big(0,T;L^\infty(\Omega)\big)$).
	\end{remark}
	
	\section*{Acknowledgements}

	\begin{small}
		The research of both authors  were partially financed by Portuguese Funds through FCT (Fundação para a Ciência e a Tecnologia) within the Projects 
		UID/04561/2025 - https://doi.org/10.54499/UID/04561/2025 (JFR) and UID/00013/2025 - https://doi.org/10.54499/UID/00013/2025 (LS), respectively. 
	\end{small}


\vspace{5mm}
	

	
	
	\begin{thebibliography}{10}
		\def\ocirc#1{\ifmmode\setbox0=\hbox{$#1$}\dimen0=\ht0
			\advance\dimen0 by1pt\rlap{\hbox to\wd0{\hss\raise\dimen0
					\hbox{\hskip.2em$\scriptscriptstyle\circ$}\hss}}#1\else
			{\accent"17 #1}\fi} 
		
		\bibitem{AzevedoRodriguesSantos2020}
		A. Azevedo J.F. Rodrigues and L. Santos: 
		Lagrange multipliers for evolution problems with constraints on the derivatives.
		Algebra i Analiz 32 no. 3 (2020)  65--83,
		St. Petersburg Math. J. 32 no. 3,   435--448 (2021) https://doi.org/10.4171/8ECM/26
		
		
		\bibitem{AzevedoRodriguesSantos2024}
		A. Azevedo J.F. Rodrigues and L. Santos:
		Nonlocal Lagrange multipliers and transport densities.
		Bull. Math. Sci.   14 no. 2, Paper No. 2350014 35 pp (2024)  https://doi.org/10.1142/S1664360723500145
		
		
		\bibitem{DS2012}
		J.C. De los Reyes, G. Stadler:
		A nonsmooth model for discontinuous shear thickening fluids: analysis and numerical solution. 
		Interfaces Free Bound. 16 (4), 575--602 (2014) https://doi.org/10.4171/IFB/330
		
		\bibitem{Eveson1995}
		S.P. Eveson:
		Compactness criteria for integral operators en $L^\infty$ and $L^1$ spaces.
		Proc. Am. Math. Soc. 123 no. 12, 3709--3716 (1995)   https://doi.org/10.2307/2161898
		
		\bibitem{HR2017}
		M. Hinterm\"uller, C. Rautenberg:
		On the uniqueness and numerical approximation of solutions to certain parabolic quasi-variational inequalities.
		Port. Math. (N.S.) 74 no. 1, 1--35  (2017)  https://doi.org/10.4171/PM/1991
		
		\bibitem{FukaoKenmochi2013}
		T. Fukao and N. Kenmochi:
		Parabolic variational inequalities with weakly time dependent constraints. Adv. Math. Sci. Appl. 23, 365--395 (2013) 
		
		\bibitem{Lions1969}
		J.L. Lions: 
		Quelques méthodes de resolution des problèmes aux limites non linéaires. Dunod, Paris, 1969
		
		\bibitem{Kenmochi2013}
		N. Kenmochi:
		Parabolic quasi-variational diffusion problems with gradient constraints. Discrete Contin. Dyn. Syst. Ser. S 6, 423--438 (2013)  https://doi.org/10.3934/dcdss.2013.6.423
		
		\bibitem{LSU1968}
		Lady{\v{z}}enskaja, O. A. and Solonnikov, V. A. and Ural'ceva, N. N.:
		Linear and quasilinear equations of parabolic type.
		Translations of Mathematical Monographs, Vol. 23 American Mathematical Society, 1968
		
		\bibitem{MNRR1996}
		J. M{\'a}lek, J. Ne{\v{c}}as, M. Rokyta, M. and M. R{\ocirc{u}}{\v{z}}i{\v{c}}ka:
		Weak and measure-valued solutions to evolutionary PDEs. Chapman \& Hall, London, 1996
		
		
		\bibitem{MirandaRodrigues2014}
		F. Miranda, J.F. Rodrigues:
		On a variational inequality for incompressible non-Newtonian
		thick flows. Recent advances in partial differential equations and applications, Contemp.
		Math., {\bf 666}, Amer. Math. Soc. Providence RI, 305--316
		(2016) https://doi.org/10.1090/conm/666
		
		\bibitem{MirandaRodriguesSantos2020}
		F. Miranda, J.F. Rodrigues and L. Santos:
		Evolutionary quasi-variational and variational inequalities with constraints on the derivatives.
		Adv. Nonlinear Anal.9 no.1, 250--277 (2020)  https://doi.org/10.1515/anona-2018-0113
		
		\bibitem{Rodrigues2013}
		J.F. Rodrigues, 
		On the Mathematical Analysis of thick fluids.
		in Boundary-value problems of mathematical physics and related problems of function theory, Part~44, Zap. Nauchn. Sem. POMI 425 (2014) 117--136 (also J. Math. Sci. (N.Y.) 210, no. 6, 835–848)(2015)) https://doi.org/10.1007/s10958-015-2594-z
		
		\bibitem{RodriguesSantos2019}
		J.F. Rodrigues and L. Santos: 
		Variational and quasi-variational inequalities with gradient type constraints. In M. Hinterm\"uller, J. F. Rodrigues (eds)	Topics in applied analysis and optimisation, CIM Ser. Math. Sci., Springer Nature Switzerland AG,  319--361 2019 
		https://doi.org/10.1007/978-3-030-33116-$0_13$
		
		\bibitem{Roubicek2013}
		T. Roubíček:
		Nonlinear Partial Differential Equations with Applications, 2nd ed., Internat. Ser. Numer. Math. 153, Birkhäuser/Springer, Basel, 2013 https://doi.org/10.1007/978-3-0348-0513-1
		
		\bibitem{Simon1987}
		J. Simon:  
		Compact sets in the space $L^{p}(0, T ; B)$.
		Ann. Mat. Pura Appl. (4), 65--96 146 (1987) https://doi.org/10.1007/BF01762360
		
	\end{thebibliography}
\end{document}